\documentclass[11pt,a4paper,twoside]{article}
\usepackage[hmarginratio=1:1,bmargin=1.3in]{geometry}
\usepackage[OT2,T1]{fontenc}
\usepackage{amsmath, amsthm}
\usepackage{amsfonts}
\usepackage{amssymb}
\usepackage[all]{xy}
\usepackage{graphicx}
\usepackage{xcolor}
\usepackage[nottoc, notlof, notlot]{tocbibind}
\usepackage{array}
\usepackage{url}
\usepackage{rotating}
\usepackage{fancyhdr}
\usepackage{titlesec}
\usepackage{xfrac} 
\usepackage{multicol}
\usepackage{comment}
\usepackage{cases}
\usepackage{hyperref}
\usepackage[normalem]{ulem}
\usepackage{enumerate}

\title{Central simple algebras, Milnor $\mathrm{K}$-theory and homogeneous spaces over complete discretely valued fields of dimension 2}
\author{Philippe Gille\footnote{CNRS, Université de Lyon 1
and Simion Stoilow Institute of Mathematics (Bucharest), \texttt{gille@math.univ-lyon1.fr}} , Diego Izquierdo\footnote{Ecole polytechnique, \texttt{diego.izquierdo@polytechnique.edu}} ~and Giancarlo Lucchini Arteche\footnote{Universidad de Chile, \texttt{luco@uchile.cl}}}
\date{}

\titleformat{\section}[hang]{\center\Large\bf}{\thesection.}{0.5cm}{}

\DeclareSymbolFont{cyrletters}{OT2}{wncyr}{m}{n}
\DeclareMathSymbol{\Sha}{\mathalpha}{cyrletters}{"58}
\DeclareMathSymbol{\Brusse}{\mathalpha}{cyrletters}{"42}

\theoremstyle{plain}
\newtheorem{theorem}{Theorem}[section]

\newtheorem{lemma}[theorem]{Lemma}
\newtheorem{proposition}[theorem]{Proposition}
\newtheorem{corollary}[theorem]{Corollary}
\newtheorem{definition}[theorem]{Definition}

\newtheorem{conj}[theorem]{Conjecture}

\theoremstyle{definition}
\newtheorem{remarque}[theorem]{Remark}

\newtheorem{example}[theorem]{Example}

\renewcommand{\cal}[1]{\mathcal{#1}}
\newcommand \spec {{\rm{Spec\,}}}

\newcommand \Br {{\rm{Br}}}

\newcommand \bD {\mathbf{D}}
\newcommand \Hom {{\mathrm {Hom}}}

\newcommand{\gm}{\mathbb{G}_\mathrm{m}}

\def\Gal{\mathop{\rm Gal}\nolimits}
\def\cO{\mathcal O}
\def\cB{\mathcal B}
\def\cA{\mathcal A}
\def\cG{\mathcal G}
\def\Stab{\mathop{\rm Stab}\nolimits}
\def\cP{\mathcal P}

\newcommand{\K}{\mathrm{K}^\mathrm{M}}
\newcommand{\kk}{\mathrm{k}^\mathrm{M}}

\newcommand{\ad}{\mathrm{ad}}
\newcommand{\sep}{\mathrm{sep}}

\newcommand \Z {{\mathbb Z}}
\newcommand \Q {{\mathbb Q}}
\newcommand \F {{\mathbb F}}
\newcommand \N {{\mathbb N}}

\renewcommand{\ker}{\mathrm{Ker}}

\newcommand \res {{\rm{Res\,}}}
\newcommand \cores {{\rm{Cores\,}}}
\newcommand{\solv}{\mathrm{solv}}
\newcommand{\car}{\mathrm{char}}

\begin{document}

\maketitle

\begin{abstract}
Let $K$ be a complete discretely valued field of dimension $2$. We prove the following statements on the arithmetic of $K$.
\begin{itemize}
\item The ``period equals index'' property holds for central simple $K$-algebras.
\item If $K$ has characteristic $0$, for every prime $p$, every class in the Milnor $\mathrm{K}$-theory modulo $p$ is represented by a symbol.
\item Serre's Conjecture II holds for the field $K$. That is, for every semisimple and simply connected $K$-group $G$, the set $H^1(K,G)$ is trivial.
\end{itemize}

\textbf{MSC Classes:} 16K50, 19F15, 11E72, 12G05, 12G10, 20G10.

\textbf{Keywords:} Period-index problem, Milnor $\mathrm{K}$-theory, Galois cohomology, principal homogeneous spaces, Serre's Conjecture II.
\end{abstract}

\section{Introduction}\label{sec intro}

In this article, we investigate various classical algebraic and arithmetic questions related to central simple algebras, Milnor $\mathrm{K}$-theory and homogeneous spaces over complete discrete valuation fields with dimension $2$. Most of the results we provide are well known when the residue field is assumed to be perfect, and our goal is to remove that assumption. This fills an important gap in the literature, since complete discrete valuation fields with imperfect residue fields are common in arithmetic settings, even when one is interested in studying characteristic zero fields. They arise, for instance, as completions of finitely generated fields.

\paragraph{Central simple algebras and the period-index problem.} The period-index problem is a classic in the theory of central simple algebras and the Brauer groups of fields. The period $\mathrm{per}(\alpha)$ of a Brauer class $\alpha$ is its order in the Brauer group and its index $\mathrm{ind}(\alpha)$ is the degree of the division algebra underlying a representative of the class. It is well known that $\mathrm{per}(\alpha)\mid\mathrm{ind}(\alpha)$ and that both have the same prime factors (cf.~for instance \cite[Thm.~2.8.7]{GS}). Relations of the type $\mathrm{ind}(\alpha)|\mathrm{per}(\alpha)^n$ are known for many fields having a nice algebraic, geometric or arithmetic behaviour (e.g. \cite{Artin,AAIKL,BenoistPeriodIndex,BHN,ColliotOjangurenParimala,DeJongPeriodIndex,FordSaltman,LieblichPeriodIndex,SaltmanPeriodIndex,ParimalaSureshPeriodIndex}). It is often reasonable to expect the equality $\mathrm{ind}(\alpha) = \mathrm{per}(\alpha)$ for Brauer classes $\alpha$ over a field with dimension $2$. From this perspective and thanks to an argument of Kato, we prove that such a result holds over any complete discrete valuation field $K$ with dimension $2$ (cf.~Proposition \ref{prop csa1} \eqref{item per=ind}). Moreover, in the case where $K$ has characteristic $0$ and contains all the $p$-th roots of unity, where $p$ is the characteristic exponent of its residue field, arguments of Saltman allow us to check a stronger result, stating that all central simple algebras over $K$ are cyclic (cf.~Proposition \ref{prop csa1} \eqref{item cyclic}).

\paragraph{Milnor $\mathrm{K}$-theory.} If $K$ is a field containing the $p$-th roots of unity for a given prime $p$, the Bloch--Kato conjecture in degree $2$ (the Merkurjev--Suslin theorem \cite{MS}, see also \cite[Thm.~8.6.5]{GS}; the general case is a theorem of Rost and Voevodsky \cite{Voevodsky}) provides an isomorphism
\[\K_2(K)/p\xrightarrow{\sim} \Br(K)[p],\]
sending the class of the symbol $\{x,y\}$ to the class of the cyclic algebra $(x,y)_{\zeta_p}$. In particular, the cyclicity of central simple algebras of period $p$ over $K$ can be restated in terms of Milnor $\mathrm{K}$-theory as every class in $\K_2(K)/p$ being represented by a symbol. It is then natural to ask whether such a result holds for complete discretely valued fields of dimension 2 without the assumption on roots of unity. The first main result of this article, proved in \S\ref{sec K-th}, shows this for fields of characteristic $0$.

\begin{theorem}[Theorem \ref{thm period index}]\label{thm period index intro}
Let $K$ be a complete discretely valued field of characteristic $0$ with residue field $\bar K$ of dimension $1$. Let $p$ be a prime number. Then every class in $\K_2(K)/p$ is represented by a symbol.
\end{theorem}

In order to tackle this statement, the main tool is a description of the $\mathrm{K}$-theory modulo $p$ for complete discretely valued fields of mixed characteristic given by Kato. However, this description works under the hypothesis that $K$ contains the $p$-th roots of unity. We remove this assumption in Proposition \ref{prop isom Kato}. Theorem \ref{thm period index intro} then readily follows by analogy with Saltman's arguments.

\paragraph{Homogeneous spaces.}
Period-index questions have important applications to the arithmetic of homogeneous spaces over low-dimensional fields, and more precisely to the so-called Serre's Conjecture II.

\begin{conj}[Serre's Conjecture II]\label{conj Serre II}
Let $K$ be a field of dimension $\leq 2$. Then every principal homogeneous space under a semisimple simply connected $K$-group has a rational point.
\end{conj}

This conjecture remains open, although a lot has been done in particular cases, whether specifying the type of the group (cf.~for instance \cite{BP,BFT,garibaldi,Invol,garibaldinote,ChernousovSerreII,GilleSerreIIqd}) or the field of definition (cf.~\cite{KneserH1LocalI, KneserH1LocalII,BT3,HarderHasseCar0I, HarderHasseCar0II,ChernousovE8,CTGP,dJHS,BenoistPeriodIndex}). For a survey on the topic, see the lecture notes \cite{GilleLNM}. More recently, in \cite{ILA}, it is proved that Serre's Conjecture II can be reduced to the case of countable characteristic $0$ fields of dimension $2$.\\

In this article, we settle the following case of Serre's Conjecture II.

\begin{theorem}[Theorem \ref{thm Serre II}]\label{thm Serre II intro}
Let $K$ be a complete discretely valued field of dimension $2$. Let $G$ be a semisimple simply connected $K$-group. Then $H^1(K,G)=1$.
\end{theorem}

Let us stress that this conjecture lends itself to a type by type analysis. Indeed, a semisimple simply connected $K$-group is a direct product of $K$-groups of the form $R_{K'/K}(G')$, where $K'/K$ is a finite separable extension and $G'$ is an absolutely almost simple simply connected $K'$-group (this is the isotypic decomposition of \cite[Exp.~XXIV, 5.3, 5.9 and 5.10]{SGA3}). Since $H^1(K,R_{K'/K}(G'))=H^1(K',G')$ by Shapiro's Lemma, it suffices to settle Conjecture \ref{conj Serre II} for absolutely almost simple simply connected groups, and hence one Dynkin type at a time.

When the residue field $\bar K$ is perfect, this result was already known by work of Kneser in the $p$-adic case \cite{KneserH1LocalI, KneserH1LocalII} and Bruhat and Tits \cite[Cor.\ 3.15]{BT3}. The cases of groups of classical type were also already known by results of Bayer-Parimala \cite{BP} and of Berhuy, Frings and Tignol \cite{BFT}, and those of types $F_4$ and $G_2$ by \cite[Cor.~6.0.3]{GilleLNM}. In \S \ref{sec Serre II}, we deduce the cases of trialitarian $D_4$, $E_6$ and $E_7$ from Proposition \ref{prop csa1}, while we obtain the case of $E_8$ by two methods: when $K$ has characteristic $0$, we study the maximal solvable extension of $K$ of degree $2^a3^b5^c$ (cf. Proposition \ref{prop Bogomolov}) and use an argument due to the first author; when $K$ has positive characteristic, we follow the original arguments by Bruhat and Tits in \cite{BT3} and adapt them to the case of imperfect residue fields.

This result has several applications. The first one we should mention is that Theorem \ref{thm Serre II} permits to extend to the $E_8$-case the 
method of Kisin--Pappas on extension of torsors \cite[\S 1.4]{KisinPappas}, see \S\ref{section_KP}. The statement obtained in this way, Proposition \ref{prop_KP}, is not new: it has been established in full generality by Ansch\"utz \cite[Prop.~10.3]{An22} by a different method. What our results provide is a short proof along the original lines of \cite{KisinPappas}. We also give applications to the study of torsors over curves in Corollaries \ref{cor_curve} and \ref{cor_line}. Finally, in the characteristic zero case, we generalize Kneser's computation of the Galois cohomology of an arbitrary semisimple group, see Theorem \ref{thm_BT+}.\\

\subsection*{Outline of the article}
In \S\ref{sec preliminaries}, we give some preliminaries on the various existing notions of dimension for fields, on central simple algebras and on Milnor $\mathrm{K}$-theory. In \S\ref{sec CSA}, we settle the ``period equals index'' property for complete discretely valued fields of dimension 2 in arbitrary characteristic. In \S\ref{sec K-th}, we work in mixed characteristic $(0,p)$, we recall Kato's computation of the graded pieces of the filtration of $\K_2(K)/p$ by the subgroups $u^i_2(K)$ when $K$ contains the $p$-th roots of unity and we remove that hypothesis, and we apply those results in order to prove Theorem \ref{thm period index intro}. In order to do so, we state and prove three ``absorption lemmas'' that help us to reduce sums of symbols to a single symbol (Lemmas \ref{lema absorcion i coprimo}, \ref{lema absorcion i divide a p} and \ref{lema absorcion i=0}). The proof of these lemmas is completely analogous to computations previously carried out by Saltman. We close \S\ref{sec K-th} with some remarks on the case where $K$ has positive characteristic. Finally, in \S\ref{sec Serre II}, we prove Theorem \ref{thm Serre II intro} separating the cases of characteristic $0$ and positive characteristic. We immediately deduce several applications, which are scattered through \S\ref{sec Serre II} and \S\ref{section_KP}.

\subsection*{Acknowledgements}

The authors would like to thank an anonymous referee for their careful and precise reading of the paper and their many suggestions that helped to vastly improve the exposition, Skip Garibaldi for his suggestions on a first version of the paper, as well as Georgios Pappas, Raman Parimala and Suresh Venapally for their interest.

The first author was supported by the project ``Group schemes, root systems, and related representations'' funded by the European Union - NextGenerationEU through Romania's National Recovery and Resilience Plan (PNRR) call no.~PNRR-III-C9-2023-I8, Project CF159/31.07.2023, and coordinated by the Ministry of Research, Innovation and Digitalization (MCID)
of Romania. The third author's research was partially supported by ANID via FONDECYT Grant 1240001.

\section{Notations and Preliminaries}\label{sec preliminaries}

The following notations concerning complete discretely valued fields will be used throughout the whole article, except in sections \S\S\ref{sec prel cd}--\ref{sec prel K-th}, where we recall some results over arbitrary fields:
\begin{itemize}
    \item[\textbullet] $K$ stands for a field that is complete with respect to a discrete valuation $v_K$;
    \item[\textbullet] $\cal O_K$ stands for its integer ring and $\bar K$ for its residue field, which can be arbitrary; 
    \item[\textbullet] for an element $x\in\cal O_K$, its image in $\bar K$ is denoted by $\bar x$.
\end{itemize}

Throughout the article, a bar always denotes reduction modulo the maximal ideal; in particular $\bar K$ always stands for the \emph{residue} field of $K$ and never for an algebraic closure. Separable and algebraic closures are denoted by $K^\sep$ and $K^{\mathrm{alg}}$ respectively.

We start by recalling some basic facts about the dimension of fields, central simple algebras and Milnor $\mathrm{K}$-theory.

\subsection{The dimension of a field}\label{sec prel cd}

The cohomological dimension $\mathrm{cd}(K)$ of a field $K$ is the cohomological dimension of its absolute Galois group. In other words, it is the smallest integer $\delta$ (or $\infty$ if such an integer does not exist) such that $H^n(K,M)=0$ for every $n>\delta$ and each finite Galois module $M$. Likewise, for a prime number $p$, the \emph{$p$-cohomological dimension} $\mathrm{cd}_p(K)$ of $K$ is the smallest integer $\delta$ (or $\infty$) such that the $p$-primary torsion subgroup of $H^n(K,M)$ vanishes for every $n>\delta$ and each finite Galois module $M$; equivalently, it is the cohomological dimension of a pro-$p$-Sylow subgroup of the absolute Galois group of $K$, cf.~\cite[I, \S3.3]{SerreCohGal}. This notion turns out not to be completely satisfactory for imperfect fields, in particular because a field with characteristic $p>0$ and whose Galois group is a pro-$p$-group has cohomological dimension at most $1$. Better definitions in that context were provided by Kato in \cite{Kato} and by the first author in \cite{gille}.

\begin{definition}\label{def cohdim}
Let $K$ be any field and $p$ a prime number. 
\begin{itemize}
\item[(i)] Assume that the characteristic of $K$ is different from $p$. The \emph{$p$-dimension} $\mathrm{dim}_{p}(K)$ and the \emph{separable $p$-dimension} $\mathrm{sdim}_{p}(K)$ of $K$ are both equal to the $p$-cohomological dimension of the absolute Galois group of $K$.
\item[(ii)] Assume now that $K$ has characteristic $p$. Let $\Omega^i_K$ be the $i$-th exterior product over $K$ of the absolute differential module $\Omega^1_{K/\mathbb{Z}}$ and let $H^{i+1}_p(K)$ be the cokernel of the morphism $\mathfrak{p}^i_K\,:\, \Omega^i_K \to \Omega^i_K/d(\Omega^{i-1}_K)$ defined by
\[x \frac{dy_1}{y_1} \wedge ... \wedge \frac{dy_i}{y_i} \mapsto (x^p-x) \frac{dy_1}{y_1} \wedge ... \wedge \frac{dy_i}{y_i} \mod d(\Omega^{i-1}_K),\]
for $x\in K$ and $y_1,...,y_i \in K^{\times}$. The \emph{$p$-dimension} $\mathrm{dim}_{p}(K)$ of $K$ is the smallest integer $\delta$ (or $\infty$ if such an integer does not exist) such that $[K:K^p]\leq p^\delta$ and $H^{\delta+1}_p(L)=0$ for all finite extensions $L$ of $K$. The \emph{separable $p$-dimension} $\mathrm{sdim}_{p}(K)$ of $K$ is the smallest integer $i$ (or $\infty$ if such an integer does not exist) such that $H^{i+1}_p(L)=0$ for all finite separable extensions $L$ of $K$.
\item[(iii)] The \emph{dimension} $\mathrm{dim}(K)$ of $K$ is the supremum of all the $\mathrm{dim}_{\ell}(K)$'s when $\ell$ runs through all prime numbers. The \emph{separable dimension} $\mathrm{sdim}(K)$ of $K$ is the supremum of all the $\mathrm{sdim}_{\ell}(K)$'s when $\ell$ runs through all prime numbers. 
\end{itemize}
Note that, for any prime number $\ell$, we have $\mathrm{sdim}_\ell(K)\leq \dim_\ell(K)$, and hence $\mathrm{sdim}(K)\leq \dim(K)$.
\end{definition}

\begin{remarque}
Let us point out here that the invariant denoted $\dim_p(K)$ in \cite[p.~220]{Kato} coincides with the one we define here. Indeed, when $p$ is the characteristic of $K$, the condition $[K:K^p]\leq p^\delta$ implies $\Omega^{\delta+1}_K=0$, and hence $H^{\delta+2}_p(L)=0$ for every finite extension $L/K$.
\end{remarque}

For the reader's convenience, we briefly recall the behaviour of these different notions of dimension.

\begin{lemma}\label{lem_skip} 
Let $K$ be a field.
\begin{enumerate}
    \item[(i)] A field $K$ has dimension $0$ if, and only if, $K$ is algebraically closed.
    \item[(ii)] Let $K$ be a field with $[K:K^p]\leq p$ where $p$ is the characteristic exponent. Then $K$ has dimension at most $1$ if, and only if, $\Br(L)=0$ for each finite extension $L/K$.
\end{enumerate}
\end{lemma}

The hypothesis $[K:K^p]\leq p$ in (ii) is not used in the proof below: it is imposed only because it is part of the requirement $\mathrm{dim}_p(K)\leq 1$ in Definition \ref{def cohdim} (ii).

\begin{proof} 
We denote by $p$ the characteristic exponent of $K$.
\begin{enumerate}
    \item[(i)] If $K$ is algebraically closed, then it is straightforward that $\mathrm{dim}(K)=0$. Conversely, assume that $\mathrm{dim}(K)=0$. Then $\mathrm{dim}_\ell(K)=0$ for each $\ell \neq p$, and so the pro-$\ell$-Sylow subgroups of $\Gal(K^\sep/K)$ are trivial. It follows that $\Gal(K^\sep/K)$ is a pro-$p$-group. The case $p=1$ is thus finished. Assuming that $p>1$, we have $[K:K^p]= 1$ so that $K$ is perfect. Furthermore, we have
    \[0=H^1_p(K)=K/\{x^p-x\mid x \in K\}=H^1(\Gal(K^\sep/K),\Z/p\Z),\]
    so that $\Gal(K^\sep/K)=1$ (see \cite[I, \S.4.1, Prop.~21]{SerreCohGal}). The field $K$ is perfect and is separably closed, hence is algebraically closed.
\item[(ii)] According to \cite[Thm. 6.1.8]{GS}, for $\ell \neq p$, the field $K$ has $\ell$-dimension at most $1$ if, and only if, $\Br(L)[\ell]=0$ for each finite extension $L/K$. It remains to check the same statement when $\ell=p>1$. This follows from the identification $\Br(L)[p]=H^2_p(L)$, cf. \cite[Thm.~9.2.4]{GS}.\qedhere
\end{enumerate}
\end{proof}

\begin{example}\label{examples_dim}
\phantom{a}
\begin{enumerate}[(a)]
\item Assume that $K$ is a $C_1$ field, that is, a field over which any projective hypersurface in $\mathbb{P}^n_K$ of degree $d \leq n$ has a rational point. If $K$ is not algebraically closed,  then $\mathrm{dim}(K) = 1$.
Indeed, according to \cite[Prop.\ 6.2.3]{GS},
 we  have $\dim_\ell(K) \leq 1$
 for each $\ell \not =p$ and also $\Br(L)=0$
 for each finite field extension $L$ of $K$. This shows that $\mathrm{dim}(K) \leq 1$ when $p=1$. 
 Now, if $p>1$, we have $H^2_p(L)= \Br(L)[p]=0$ for each finite field extension $L$ of $K$ by \cite[Thm. 9.2.4]{GS} and $[K:K^p] \leq p$  by \cite[II, \S3.2, Cor.~to Prop.~8] {SerreCohGal}, so that $\mathrm{dim}_p(K) \leq 1$.
 Since $K$ is not algebraically closed, we conclude that $\mathrm{dim}(K)= 1$.
\item By taking into account respectively theorems
by Chevalley, Tsen and Lang (see \cite[II, \S3.3]{SerreCohGal}), we have $\mathrm{dim}(K)=1$ in the following cases:
\begin{itemize}
\item $K$ is a finite field;
\item $K=K_0(C)$ is the function field of an algebraic curve $C$ over an algebraically closed field $K_0$;
\item $K$ is a henselian discretely valued field such that its completion $\widehat K$ is separable over $K$ and whose residue field $\bar K$ is algebraically closed.
\end{itemize}
\item There are more  examples of fields of dimension $\leq 1$, see for example the field constructed 
by Colliot-Th\'el\`ene and Madore \cite[Thm.~1.1]{CTM}. 
\item Assume that $K$ is separably closed
of characteristic $p\geq 2$ and that $[K:K^p]=p$.
Then $\mathrm{dim}(K)=1$. Indeed each finite field extension $L$ of 
$K$ is separably closed so that $\Br(L)=0$.
\item A very important feature of Definition \ref{def cohdim} is that, for any complete discretely valued field $K$ with residue field $\bar K$, we have $\dim(K)=\dim(\bar K)+1$, cf.~\cite[Cor.~to Thm.~3]{Kato}.
\end{enumerate}
\end{example}

\subsection{Central simple algebras}\label{sec csa0}

Let $K$ be a field and let $A$ be a central simple $K$-algebra, with underlying central division algebra $D$. Recall that the \emph{period} $\mathrm{per}(A)$ of $A$ is the order of its class in the Brauer group $\Br(K)\cong H^2(K,\mathbb{G}_\mathrm{m})$, and that its \emph{index} $\mathrm{ind}(A)$ is the square root of the integer $\dim_K D$. Both invariants depend only on the class of $A$ in $\Br(K)$, so that one may speak of the period and of the index of a Brauer class. It is a standard fact that $\mathrm{ind}(A)\mid\mathrm{per}(A)$ and that both numbers share the same prime factors.\\

To a character $\chi \in \mathrm{Hom}_\mathrm{cont}(\Gal(K^\sep/K),\mathbb{Q}/\mathbb{Z})$ and an element $b \in K^\times$ one can associate a central simple algebra denoted by $(\chi,b)$ (cf. \cite[\S 2.5]{GS}). In Galois cohomology terms this corresponds to the cup-product of $\chi \in H^1(K,\mathbb{Q}/\mathbb{Z})\cong H^2(K,\mathbb{Z})$ and $b\in H^0(K,\mathbb{G}_\mathrm{m})$. The algebra $(\chi,b)$ can be easily described when the character $\chi$ has order $m$ with $m$ coprime to the characteristic exponent of $K$ and $K$ contains all the $m$-th roots of unity. In that case, one can find an element $a \in K^\times$ such that $\ker (\chi)= \Gal (K^\sep/K(\sqrt[m]{a}))$ and a primitive $m$-th root of unity $\zeta\in K$ such that $\sigma(\sqrt[m]{a})=\zeta\sqrt[m]{a}$ for every $\sigma$ lying in the coset $\chi^{-1}\!\left(\frac{1}{m}\right)$. The central simple algebra $(\chi,b)$ can then be described by generators and relations as
$$ \langle x,y : x^m=a, y^m=b, xy=\zeta yx \rangle .$$ One often writes $(a,b)_\zeta$ instead of $(\chi,b)$.\\

Central simple algebras of the form $(\chi,b)$ are usually called \emph{cyclic algebras}. If $\chi$ has order $m$, then $(\chi,b)$ has degree $m$, so that $\mathrm{ind}\bigl((\chi,b)\bigr)$ divides $m$. In particular, if $m=\ell$ is a prime number, then $(\chi,b)$ is either split or a division algebra: in the first case its period and its index are both equal to $1$, and in the second case they are both equal to $\ell$. Hence period equals index for cyclic algebras of prime degree, which is the only case we shall use in the sequel.

\subsection{Milnor $\mathrm{K}$-theory}\label{sec prel K-th}
Let $K$ be any field. We refer to \cite[Ch.~7]{GS} for a systematic treatment of the Milnor $\mathrm{K}$-theory groups of $K$; unless otherwise stated, every statement presented in this subsection can be found there. In this article, we will only use the first three Milnor $\mathrm{K}$-theory groups, which can be explicitly described as follows:
\begin{gather*}
    \K_0(K):=\mathbb{Z}, \quad \K_1(K):=K^\times\\
    \K_2(K):=(K^\times \otimes_{\mathbb{Z}} K^\times)/\left\langle x\otimes (1-x): x \in K\smallsetminus \{0,1\} \right\rangle .
\end{gather*}
For $x,y \in K^{\times}$, the symbol $\{x,y\}$ denotes the class of $x\otimes y$ in $\K_2(K)$.

When $L$ is a finite extension of $K$ and $q \in \{0,1,2\}$, we have a norm homomorphism
\[N_{L/K}: \K_q(L) \rightarrow \K_q(K),\]
satisfying the following properties:
\begin{itemize}
\item[$\bullet$] For $q=0$, the map $N_{L/K}: \K_0(L) \rightarrow \K_0(K)$ is given by multiplication by $[L:K]$.
\item[$\bullet$] For $q=1$, the map $N_{L/K}: \K_1(L) \rightarrow \K_1(K)$ coincides with the usual norm $L^{\times} \rightarrow K^{\times}$.
\item[$\bullet$] The equality $N_{L/K}(\{x,y\})=\{x,N_{L/K}(y)\}$ holds for $x \in \K_1(K)$ and $y\in \K_1(L)$. 
\end{itemize}

Milnor $\mathrm{K}$-theory is also endowed with residue maps. Indeed, when $K$ is a complete discretely valued field with the ring of integers $\cal O_{K}$ and residue field $\bar K$, there exists a unique residue morphism
\[\partial: \K_{2}(K) \rightarrow \K_{1}(\bar K)\]
such that, for any uniformizer $\pi$ and for each unit $u\in \cal O_{K}^{\times}$, one has
\[\partial (\{\pi, u\})=\bar u.\]
The kernel of $\partial$ is the subgroup $U_{2}(K)$ of $\K_{2}(K)$ generated by symbols of the form $\{x,y\}$ with $x,y \in \cal O_{K}^\times$. For $i\geq 1$ we denote by $U_{2}^i(K)$ the subgroup of $\K_{2}(K)$ generated by the symbols $\{x,y\}$ with $x \in 1+\pi^i\cal O_{K}$ and $y \in K^{\times}$. Note that $U_2^i(K)\subseteq U_2(K)$ automatically. Indeed, such an $x$ is a unit with $\bar x=1$, and writing $y=\pi^n u$ with $u\in\cal O_K^\times$ we get $\{x,y\}=n\{x,\pi\}+\{x,u\}$, whose residue is $\bar x^{-n}=1$.\\

The residue map turns out to be very useful to understand the group $\K_2(K)/p$ when $p$ is a prime number different from the characteristic of $\bar K$. When $p$ is equal to $\mathrm{char}(\bar K)$ and $K$ contains a primitive $p$-th root of unity $\zeta_p$, one can use instead a precise description of $\K_2(K)/p$ due to Kato \cite[Thm.~2]{Kato}. For that purpose, we set $\kk_2(K):=\K_2(K)/p$ and we let $u_2^i(K)$ be the image of $U_2^i(K)$ in $\kk_2(K)$, and we define $\epsilon_{K}:=\frac{v_{K}(p)p}{p-1}$. Kato's description computes the graded pieces of the filtration $\kk_2(K)\supseteq u_2^1(K)\supseteq u_2^2(K)\supseteq\cdots$ in terms of differential forms over $\bar K$. Rather than stating it here and restating it afterwards in the generality that we need, we shall state it once and for all, without the hypothesis $\zeta_p\in K$, in Proposition \ref{prop isom Kato} below; Kato's isomorphisms themselves are recalled as \eqref{eq Kato i original}--\eqref{eq Kato iv original} at the beginning of its proof, where they are used.

\section{Period and index of central simple algebras}\label{sec CSA}

In this section, we study central simple algebras over complete discretely valued fields of dimension $2$. We first record the well-known reduction of the ``period equals index'' property to the case of classes of prime period.

\begin{lemma}\label{lem reduction prime period}
Let $K$ be a field and let $\ell$ be a prime number. Assume that for every finite extension $L/K$, every class of period $\ell$ in $\Br(L)$ has index $\ell$. Then $\mathrm{per}(\alpha)=\mathrm{ind}(\alpha)$ for every $\alpha\in\Br(K)\{\ell\}$. Consequently, if the assumption holds for every prime number $\ell$, then period equals index for every central simple $K$-algebra.
\end{lemma}

\begin{proof}
Let $\alpha\in\Br(K)\{\ell\}$ and write $\mathrm{per}(\alpha)=\ell^n$; we argue by induction on $n$, the cases $n=0$ and $n=1$ being respectively trivial and part of the hypothesis. Assume $n\geq 2$. The class $\ell^{n-1}\alpha$ has period $\ell$, hence index $\ell$ by hypothesis, so that it is represented by a central division algebra $D$ of degree $\ell$. Let $L$ be a maximal subfield of $D$, so that $[L:K]=\ell$. Since $\ell^{n-1}\alpha_L=(\ell^{n-1}\alpha)_L=0$, the period of $\alpha_L$ divides $\ell^{n-1}$. The induction hypothesis yields then that $\mathrm{ind}(\alpha_L)=\mathrm{per}(\alpha_L)$ divides $\ell^{n-1}$. Let now $M$ be a maximal subfield of the division algebra underlying $\alpha_L$, so that $M$ splits $\alpha$ and $[M:K]=[M:L][L:K]$ divides $\ell^{n}$. Therefore $\mathrm{ind}(\alpha)$ divides $\ell^n=\mathrm{per}(\alpha)$ (see \cite[Prop.~4.5.1]{GS}), and since the period always divides the index we conclude that $\mathrm{ind}(\alpha)=\mathrm{per}(\alpha)$. The last assertion follows because the period and the index of a class are the products of the periods and of the indices of its primary components (see \cite[Prop.~2.8.13]{GS} for the index).
\end{proof}

The next result is essentially \cite[XII, \S3, Exer.~2, 3]{SerreCorpsLocaux}, for which we were unable to find a published proof.

\begin{lemma}\label{lem unramified Brauer}
Let $K$ be a complete discretely valued field with residue field $\bar K$ of characteristic exponent $p\geq 1$, let $K^{\mathrm{unr}}$ be the maximal unramified extension of $K$ and put $\Gamma:=\Gal(K^{\mathrm{unr}}/K)$, which is canonically isomorphic to $\Gal(\bar K^\sep/\bar K)$. Fix a uniformizer $\pi$ of $K$. Then the following hold.
\begin{enumerate}[(i)]
\item There is a short exact sequence of abelian groups
\[0\to \Br(\bar K)\to \Br(K^{\mathrm{unr}}/K)\xrightarrow{\ \partial\ } H^1(\bar K,\Q/\Z)\to 0,\]
which is split by the homomorphism $H^1(\bar K,\Q/\Z)\to \Br(K^{\mathrm{unr}}/K)$, $\chi\mapsto (\chi,\pi)$, where a character of $\Gal(\bar K^\sep/\bar K)$ is viewed as an unramified character of $K$ through the isomorphism $\Gamma\cong\Gal(\bar K^\sep/\bar K)$.
\item For any finite extension $L/K$ with ramification degree $e_{L/K}$, there is a commutative diagram
\[\xymatrix{
\Br(K^{\mathrm{unr}}/K) \ar[r]^{\partial_K} \ar[d]_{\res_{L/K}} & H^1(\bar K,\Q/\Z) \ar[d]^{e_{L/K}\res_{\bar L/\bar K}} \\
\Br(L^{\mathrm{unr}}/L) \ar[r]^{\partial_L} & H^1(\bar L,\Q/\Z).
}
\]
\item If $\ell$ is a prime number with $\ell\neq p$, then $\Br(K^{\mathrm{unr}}/K)\{\ell\}=\Br(K)\{\ell\}$. The same holds for $\ell=p$ if $\bar K$ is perfect.
\end{enumerate}
\end{lemma}

\begin{proof}
\,
\begin{enumerate}[(i)]
\item Since $K^{\mathrm{unr}}/K$ is unramified, the valuation $v_K$ extends to a surjection $v:(K^{\mathrm{unr}})^\times\to\Z$ with kernel $\cal O_{K^{\mathrm{unr}}}^\times$, and the resulting short exact sequence of $\Gamma$-modules
\[1\to \cal O_{K^{\mathrm{unr}}}^\times\to (K^{\mathrm{unr}})^\times\xrightarrow{\ v\ }\Z\to 0\]
is split by the $\Gamma$-equivariant section $n\mapsto \pi^n$, the element $\pi$ being fixed by $\Gamma$. Taking Galois cohomology, we get a split exact sequence
\[0\to H^2(\Gamma,\cal O_{K^{\mathrm{unr}}}^\times)\to \Br(K^{\mathrm{unr}}/K)\to H^2(\Gamma,\Z)\to 0.\]
We have $H^2(\Gamma,\Z)\cong H^1(\Gamma,\Q/\Z)=H^1(\bar K,\Q/\Z)$, the isomorphism being the boundary map $\delta$ of the exact sequence $0\to\Z\to\Q\to\Q/\Z\to 0$ (recall that $H^i(\Gamma,\Q)=0$ for $i\geq 1$ since $\Q$ is uniquely divisible). Since the class of the cyclic algebra $(\chi,b)$ is by definition the cup-product $\delta(\chi)\cup b$, the section provided above is precisely $\chi\mapsto \delta(\chi)\cup\pi=(\chi,\pi)$, and it is a group homomorphism. This also shows that the projection $\Br(K^{\mathrm{unr}}/K)\to H^1(\bar K,\Q/\Z)$ is the residue map $\partial$.

It remains to identify $H^2(\Gamma,\cal O_{K^{\mathrm{unr}}}^\times)$ with $\Br(\bar K)$. Reduction modulo the maximal ideal gives a short exact sequence of $\Gamma$-modules
\[1\to U^{(1)}\to \cal O_{K^{\mathrm{unr}}}^\times\to (\bar K^\sep)^\times\to 1,\]
where $U^{(1)}:=1+\mathfrak{m}_{K^{\mathrm{unr}}}$. For a finite unramified Galois extension $L/K$, denote $U^{(n)}_L:=1+\mathfrak m_L^n$. For such an $L$ and for every $n\geq 1$, we have a Galois-equivariant isomorphism $U^{(n)}_L/U^{(n+1)}_L\cong \bar L^{+}$, hence $H^i(L/K,U^{(n)}_L/U^{(n+1)}_L)=0$ for every $i\geq 1$. Then \cite[XII, \S3, Lem.~3]{SerreCorpsLocaux} yields $H^i(L/K,U^{(1)}_L)=0$ and consequently $H^i(\Gamma,U^{(1)})=0$ for $i\geq 1$. This gives the desired isomorphism $H^2(\Gamma,\cal O^\times_{K^{\mathrm{unr}}})\cong H^2(\Gamma,(\bar K^\sep)^\times)=\Br(\bar K)$.

\item Since the image of the subgroup $\Br(\bar K)\cong H^2(\Gamma,\cal O^\times_{K^{\mathrm{unr}}})\subseteq \Br(K^{\mathrm{unr}}/K)$ via $\res_{L/K}$ clearly factors through $\Br(\bar L)$ and the exact sequence from (i) is split, it suffices to check the commutativity for an element of the form $(\chi,\pi)\in\Br(K^{\mathrm{unr}}/K)$. Writing $\pi=\varpi^{e_{L/K}}u$ with $\varpi$ a uniformizer of $L$ and $u\in\cal O_L^\times$, we see that
\[\partial_L(\res_{L/K}(\chi,\pi))=\partial_L(e_{L/K}(\res_{\bar L/\bar K}(\chi),\varpi)+(\res_{\bar L/\bar K}(\chi),u))=e_{L/K}\res_{\bar L/\bar K}(\chi),\]
where the last equality uses the fact that $(\res_{\bar L/\bar K}(\chi),u)\in H^2(\Gamma_L,\cal O_{L^\mathrm{unr}}^\times)\cong\Br(\bar L)$.\\

\item It suffices to show that $\Br(K^{\mathrm{unr}})\{\ell\}=0$. The field $K^{\mathrm{unr}}$ is henselian, discretely valued, with separably closed residue field $\bar K^\sep$. Hence every finite separable extension of $K^{\mathrm{unr}}$ of degree prime to $p$ is totally and tamely ramified, therefore cyclic of the form $K^{\mathrm{unr}}(\sqrt[n]{\pi})$; it follows that a pro-$\ell$-Sylow subgroup of $\Gal((K^{\mathrm{unr}})^\sep/K^{\mathrm{unr}})$ is isomorphic to $\Z_\ell$ for $\ell\neq p$, so that $\mathrm{cd}_\ell(K^{\mathrm{unr}})\leq 1$ and $\Br(K^{\mathrm{unr}})\{\ell\}=0$ by \cite[Thm.~6.1.8]{GS}. If $\bar K$ is perfect, then $K^{\mathrm{unr}}$ has dimension $1$ by Example \ref{examples_dim} (b), whence $\Br(K^{\mathrm{unr}})=0$.
\end{enumerate}
\end{proof}

We can now state and prove the main result of this section.

\begin{proposition}\label{prop csa1}
Let $K$ be a complete discretely valued field with residue field $\bar K$ of dimension $1$ (equivalently, $\dim(K)=2$) and of characteristic exponent $p\geq 1$. Then the following hold.
\begin{enumerate}[(i)]
\item\label{item per=ind} Period equals index for every central simple $K$-algebra.
\item\label{item cyclic} If moreover $\mathrm{char}(K)=0$ and $\zeta_p\in K$, then every central simple $K$-algebra of prime period is cyclic.
\end{enumerate}
\end{proposition}

\begin{proof}
We denote by $K^{\mathrm{unr}}$ the maximal unramified extension of $K$, we fix a prime number $\ell$ and we let $\alpha\in\Br(K)[\ell]$. By Lemma \ref{lem reduction prime period}, the proof of statement \eqref{item per=ind} is reduced to the same statement for $\alpha$. We shall prove then that $\alpha$ has index $\ell$ and, under the assumptions of \eqref{item cyclic}, that it is represented by a cyclic algebra of degree $\ell$.

\emph{Case 1: $\alpha$ is split by $K^{\mathrm{unr}}$.} Since $\dim(\bar K)=1$ we have $\Br(\bar K)=0$ by Lemma \ref{lem_skip} (ii), so that Lemma \ref{lem unramified Brauer} (i) provides an isomorphism $\partial:\Br(K^{\mathrm{unr}}/K)\xrightarrow{\sim}H^1(\bar K,\Q/\Z)$ whose inverse is $\chi\mapsto(\chi,\pi)$, for any fixed uniformizer $\pi$ of $K$. In particular $\alpha=(\chi,\pi)$ for an unramified character $\chi$ of order $\ell$, that is, $\alpha$ is represented by a cyclic algebra of degree $\ell$, whose index is $\ell$ by \S\ref{sec csa0}.

\emph{Case 2: $\alpha$ is not split by $K^{\mathrm{unr}}$.} By Lemma \ref{lem unramified Brauer} (ii) this forces $\ell=p>1$ and $\bar K$ imperfect; since $\dim_p(\bar K)\leq 1$ we then have $[\bar K:\bar K^p]=p$. Assume first that we are in the situation of \eqref{item cyclic}, that is, $\mathrm{char}(K)=0$ and $\zeta_p\in K$. By \cite[Prop.~2.8]{LieblichParimalaSuresh} there are a uniformizer $\pi$ of $K$ and an element $a\in K^\times$ such that $\alpha-(a,\pi)_{\zeta_p}\in\Br(K^{\mathrm{unr}}/K)$. Using Case 1 we may write $\alpha-(a,\pi)_{\zeta_p}=(\chi,\pi)=(b,\pi)_{\zeta_p}$ for some $b\in K^\times$, and hence
\[\alpha=(a,\pi)_{\zeta_p}+(b,\pi)_{\zeta_p}=(ab,\pi)_{\zeta_p},\]
so that $\alpha$ is represented by a cyclic algebra of degree $p$, as desired. In the general case, that is, without assuming $\mathrm{char}(K)=0$ or $\zeta_p\in K$, Kato has shown that $\alpha$ has index $p$. More precisely, \cite[\S 4, Lem.~5]{Kato1979} states that the division algebra $D$ representing $\alpha$ satisfies the two assumptions made at the beginning of \cite[\S 1]{Kato1979}, the first of which reads $\dim_K D=p^2$, that is, $\mathrm{ind}(\alpha)=p$.

Finally, if $\bar K$ is perfect (in particular if $p=1$) then Case 2 does not occur by Lemma \ref{lem unramified Brauer} (ii), and Case 1 gives the conclusion for every prime $\ell$.
\end{proof}

\section{Symbols and Milnor $\mathrm{K}$-theory}\label{sec K-th}

In this section, we are interested in the second degree Milnor $\mathrm{K}$-theory group of a complete discretely valued field $K$ of dimension $2$. The main theorem we prove generalizes the cyclicity result in Proposition \ref{prop csa1} \eqref{item cyclic} to Milnor $\mathrm{K}$-theory classes when $K$ is not anymore assumed to contain all the $p$-th roots of unity, and it can be stated as follows.

\begin{theorem}\label{thm period index}
Let $K$ be a complete discretely valued field of characteristic $0$ with residue field $\bar K$ of dimension $\leq 1$ (equivalently, $\dim(K)\leq 2$). Let $p$ be a prime number. Then every class in $\K_2(K)/p$ is represented by a symbol.
\end{theorem}

Let us explain in which sense this is a generalisation of Proposition \ref{prop csa1} \eqref{item cyclic}. Recall from \S\ref{sec intro} that if $\zeta_p\in K$, the Galois symbol
\[\K_2(K)/p\xrightarrow{\ \sim\ }H^2(K,\mu_p^{\otimes 2})\cong \Br(K)[p]\]
is an isomorphism, the second identification being given by a choice of $\zeta_p$. Under this isomorphism the symbol $\{x,y\}$ corresponds to the class of the cyclic algebra $(x,y)_{\zeta_p}$ (cf.~\cite[Prop.~4.7.1]{GS}). Hence, when $\zeta_p\in K$, saying that every class of $\K_2(K)/p$ is represented by a symbol is exactly saying that every class of $\Br(K)[p]$ is cyclic. Theorem \ref{thm period index} removes the hypothesis $\zeta_p\in K$, for which there is no such dictionary available.\\

The rest of this section is devoted to the proof of this theorem. If $\bar K$ has dimension $0$, then $K$ has dimension $1$ by Example \ref{examples_dim} (e). Since $\K_2(K)/p$ is isomorphic to $H^2(K,\mu_p^{\otimes 2})$, it vanishes because $K$ has dimension $1$. In particular, all of its elements are symbols. We may therefore assume that $\bar K$ has dimension $1$. Let us first deal with the easy case when $\bar K$ has characteristic $\neq p$. Under that assumption, since $p$ is invertible in $\bar K$, by \cite[Cor.~7.1.10]{GS}, we have the exact sequence
$$0 \rightarrow \K_2(\bar K)/p \rightarrow \K_2(K)/p \xrightarrow{\partial} \K_1(\bar K)/p \rightarrow 0,$$
where the residue map $\partial: \K_2(K)/p \rightarrow \K_1(\bar K)/p$ sends a symbol of the form  $\{\pi,u\}\in \K_2(K)/p$, with $\pi$ a uniformizer in $K$ and $u$ a unit in $\mathcal{O}_K$, to $\overline{u}\in \K_1(\bar K)/p$. Since $\bar K$ has dimension $1$, the group $\K_2(\bar K)/p$ vanishes. Hence $\partial$ is an isomorphism, and every element in $\K_2(K)/p$ is a symbol of the form $\{\pi,u\}$.\\

Let us now move on to the case where $\mathrm{char}(\bar K)=p$. Recall from \S\ref{sec prel K-th} that in this context we write $\kk_i(K):=\K_i(K)/p$ and $u_2^i(K)$ denotes the image of $U_2^i(K)$ in $\kk_2(K)$.

If $K$ contains all the $p$-th roots of unity, Theorem \ref{thm period index} immediately boils down to Proposition \ref{prop csa1} \eqref{item cyclic} by the dictionary recalled above. We henceforth do not assume anymore that $K$ contains all the $p$-th roots of unity. In that case, we a priori do not have Kato's isomorphisms describing the graded pieces of the filtration $\kk_2(K)\supseteq u_2^1(K)\supseteq u_2^2(K)\supseteq \cdots$, so we first need to extend them.

The following elementary lemma explains why the filtration $(u_2^i(K))_{i\geq 0}$ stops at $\epsilon_K$. This fact is used repeatedly below, and it is stated without proof in \cite[Thm.~2]{Kato}, so we include an argument.

\begin{lemma}\label{lem vanishing above eps}
Let $K$ be a complete discretely valued field of characteristic $0$ whose residue field has characteristic $p$, put $\epsilon_K:=\frac{pv_K(p)}{p-1}$. Then $1+a\in (K^\times)^p$ for every $a\in K$ with $v_K(a)>\epsilon_K$. Consequently $U_2^i(K)\subseteq p\,\K_2(K)$ and $u_2^i(K)=0$ for every integer $i>\epsilon_K$.
\end{lemma}

\begin{proof}
The power series for $p$-th roots reads
\[(1+a)^{1/p}=\sum_{i=0}^\infty \binom{1/p}{i} a^i.\]
For each term in the series we have
\[v_K\left(\binom{1/p}{i}a^i\right)=v_K\left(\binom{1/p}{i}\right)+iv_K(a)=v_p\left(\binom{1/p}{i}\right)v_K(p)+iv_K(a),\]
where $v_p$ stands for the usual $p$-adic valuation. Since $v_K(a)=\epsilon_K+\delta$ for some $\delta>0$, we get
\begin{align*}
v_K\left(\binom{1/p}{i}a^i\right) &= v_p\left(\binom{1/p}{i}\right)v_K(p)+\frac{ipv_K(p)}{p-1}+i\delta\\
&=v_K(p)\left(v_p\left(\binom{1/p}{i}\right)+\frac{ip}{p-1}\right)+i\delta.
\end{align*}
Since $v_p(\binom{1/p}{i})\geq -\frac{ip}{p-1}$, the valuation tends to infinity when $i$ grows. We see then that the series converges in $K$ and hence $1+a$ admits a $p$-th root in $K$.

For the second assertion, let $i>\epsilon_K$ be an integer and let $\{x,y\}$ be a generator of $U^i_2(K)$, so that $x\in 1+\pi^i\cal O_K$. By what precedes $x\in (K^\times)^p$, whence $\{x,y\}\in p\,\K_2(K)$.
\end{proof}

\begin{proposition}\label{prop isom Kato}
Let $K$ be a complete discretely valued field of characteristic $0$ with residue field $\bar K$ of dimension $1$ (equivalently, $\dim(K)=2$) and characteristic $p$. Define $\epsilon_{K}:=\frac{pv_{K}(p)}{p-1}$. Then $u_2^{\lfloor\epsilon_K\rfloor +1}(K)$ is trivial and we have the following isomorphisms for any uniformizer $\pi\in K$:
\begin{align}
\label{eq Kato i}
\rho_0 : \kk_1(\bar K) &\xrightarrow{\sim} \kk_{2}(K)/ u_{2}^1(K), \\
\nonumber \bar y &\mapsto \{y,\pi\}; \\
\intertext{for every $0<i<\epsilon_{K}$ with $(i,p)=1$ and any $\omega$ with $v_K(\omega)=i$,}
\label{eq Kato ii}
\rho_i: \Omega^{1}_{\bar K} &\xrightarrow{\sim} u_{2}^i(K)/ u_{2}^{i+1}(K), \\
\nonumber \bar x \,\frac{d\bar y}{\bar y} &\mapsto \{1+\omega x,y\}; \\
\intertext{for every $0<i<\epsilon_{K}$ with $p\mid i$ and any $\varpi$ with $v_K(\varpi)=i/p$,}
\label{eq Kato iii}
\rho_i : \Omega^{0}_{\bar K}/\ker(d) &\xrightarrow{\sim}  u_{2}^i(K)/ u_{2}^{i+1}(K), \\
\nonumber \bar x\pmod{\bar K^p} &\mapsto  \{1+\varpi^p x,\pi\}; \\
\intertext{and for $i=\epsilon_{K}$ when $\epsilon_K\in\Z$ and every $\omega$ with $v_K(\omega)=i$},
\label{eq Kato iv}
\rho_{i} : \bar K/(\bar K \cap \overline{u}\Phi(\bar M)) &\xrightarrow{\sim} u_{2}^{i}(K), \\
\nonumber  \bar x &\mapsto  \{1+\omega x,\pi\}
\end{align}
where the field $M$ stands for $K(\zeta_p)$, $u=(\zeta_p-1)^p/\omega\in\cal O_M^\times$ and $\Phi(\bar M):=\{t^p-t: t \in \bar M\}$.
\end{proposition}

\begin{remarque}\label{rem H^1_p}
When $\zeta_p\in K$, i.e., when $M=K$, the isomorphism \eqref{eq Kato iv} recovers Kato's isomorphism by taking $\omega=(\zeta_p-1)^p$. Indeed, then $\bar u=1$ and the quotient $\bar K/(\bar K \cap\Phi(\bar K))$ is precisely $H^1_p(\bar K)$.

More generally, when $M/K$ is totally ramified, we have $\bar u \in \bar M=\bar K$, hence the map
$$H^1_p(\bar K)=\bar K/\Phi(\bar K)\rightarrow \bar K/(\bar K \cap \bar u \Phi(\bar K))\quad \text{given by}\quad \bar x \mapsto \bar u \bar x$$ is an isomorphism and $\rho_{\epsilon_{K}}$ induces an isomorphism between $H^1_p(\bar K)$ and $u_2^i(K)$.

On the other hand, when $M/K$ is unramified, we may have that $\bar K/(\bar K \cap \overline{u}\Phi(\bar M))$ is not isomorphic to $H^1_p(\bar K)$, even when $\bar K$ is perfect. For instance, if $K=\Q_3(\sqrt{-6})$, then $\bar K=\F_3$ and $\bar M=\F_9$. A direct computation gives that
\[H_p^1(\bar K)=\bar K/\Phi(\bar K)=\F_3.\]
Now, if we choose $\pi=\sqrt{-6}$ as a uniformizer and $\omega:=\pi^3$, we get that $\bar u=\sqrt{2}^{-1}\in\bar M$. Since $\Phi(\sqrt{2})=\sqrt{2}$, we get $1\in\bar u \Phi(\bar M)$ and hence
\[\bar K/(\bar K\cap \bar u\Phi(\bar M))=\bar K/\bar K=0.\]
\end{remarque}

\begin{proof}[Proof of Proposition \ref{prop isom Kato}]
Recall first Kato's theorem \cite[Thm.~2]{Kato}: if $K$ is a complete discretely valued field of characteristic $0$, with residue field $\bar K$ of characteristic $p$, and if $\zeta_p\in K$, then for any uniformizer $\pi\in K$ there are isomorphisms
\begin{align}
\label{eq Kato i original}
\rho^{2}_0 : \kk_{2}(\bar K)\oplus \kk_1(\bar K) &\xrightarrow{\sim} \kk_{2}(K)/ u_{2}^1(K), \\
\nonumber\left(\{\bar x,\bar y\},0\right) &\mapsto \{x,y\}, \\
\nonumber\left(0,\{\bar y\}\right) &\mapsto \{\pi, y\}; \\
\intertext{for every $0<i<\epsilon_{K}$ such that $(i,p)=1$ and any $\omega$ with $v_K(\omega)=i$,}
\label{eq Kato ii original}
\rho^{2}_i: \Omega^{1}_{\bar K} &\xrightarrow{\sim} u_{2}^i(K)/ u_{2}^{i+1}(K), \\
\nonumber \bar x \,\frac{d\bar y}{\bar y} &\mapsto \{1+\omega x, y\}; \\
\intertext{for every $0<i<\epsilon_{K}$ such that $p\mid i$ and any $\varpi$ with $v_K(\varpi)=i/p$, }
\label{eq Kato iii original}
\rho^{2}_i : \Omega^{1}_{\bar K}/\ker(d) \oplus \Omega^{0}_{\bar K}/\ker(d) &\xrightarrow{\sim}  u_{2}^i( K)/ u_{2}^{i+1}( K), \\
\nonumber\left( \bar x \,\frac{d\bar y}{\bar y}, 0 \right) &\mapsto  \{1+\varpi^p x, y\}, \\
\nonumber\left( 0, \bar x\right) &\mapsto  \{\pi,1+\varpi^p x\}; \\
\intertext{for $i=\epsilon_{K}$,}
\label{eq Kato iv original}
\rho^{2}_{i} : H_p^{2}(\bar K)\oplus H_p^{1}(\bar K) &\xrightarrow{\sim} u_{2}^{i}( K), \\
\nonumber\left( \bar x \,\frac{d\bar y}{\bar y}, 0 \right) &\mapsto  \{1+(\zeta_p-1)^p x, y\}, \\
\nonumber\left( 0, \bar x \right) &\mapsto  \{\pi,1+(\zeta_p-1)^p x\}.
\end{align}

Since $\dim(\bar K)=1$, for any $\bar y\in \bar K^\times\smallsetminus (\bar K^\times)^p$ we have
\[\Omega^0_{\bar K}=\bar K,\quad \Omega^1_{\bar K}=\bar K\frac{d\bar y}{\bar y},\quad\text{and}\quad \Omega^2_{\bar K}=0.\]
This implies that the first summand in the source vanishes in \eqref{eq Kato i original}, \eqref{eq Kato iii original} and \eqref{eq Kato iv original} (use the Bloch--Gabber--Kato theorem \cite[Thm.~9.5.2]{GS} for \eqref{eq Kato i original}). Thus, when $\zeta_p\in K$, the isomorphisms from the statement are immediately given by \eqref{eq Kato i original}--\eqref{eq Kato iv original} and Remark \ref{rem H^1_p}, except that we have reversed the order of the coordinates in some symbols; since $\{a,b\}=-\{b,a\}$ by \cite[Prop.~7.1.1]{GS}, this replaces $\rho_i$ by $-\rho_i$ and is therefore harmless. Note that this makes isomorphism \eqref{eq Kato iv} available with an arbitrary $\omega$ instead of $(\zeta_p-1)^p$ even when $\zeta_p\in K$.\\

We assume now that $\zeta_p\not\in K$. Define $M:=K(\zeta_p)$ and let $L/K$ be the maximal unramified subextension of $M/K$. Then $M/L$ is a totally ramified cyclic extension of degree $d$ dividing $p-1$, and there exists a uniformizer $\pi\in K$ such that $M=L(\sqrt[d]{\pi})$. Furthermore, the isomorphisms \eqref{eq Kato i}--\eqref{eq Kato iv} are valid over $M$, by the case already treated. We will denote these isomorphisms by $\rho_i^M$.\\

Let us assume first that $L=K$. Write $u_2^0(K):=\kk_2(K)$ and note that, for every $i\geq 0$, $\res_{M/K}(u_2^i(K))\subseteq u_2^{di}(M)$. Thus, there is a natural map
\begin{equation}\label{eq restriction}
\res_{M/K}:u_2^i(K)/u_2^{i+1}(K)\to u_2^{di}(M)/u_2^{di+1}(M),
\end{equation}
which we claim to be an isomorphism for every $0\leq i\leq\epsilon_K$. Moreover, up to multiplication by an integer prime to $p$, its inverse is given by the norm map $N_{M/K}$.

Let us prove surjectivity. This amounts to proving that (the class of) the symbol on the right hand side of $\rho^M_{di}$ is in the image of the restriction morphism. Since these isomorphisms are well defined regardless of the choice of a uniformizer $\pi_M$, of the elements $\omega,\varpi\in M$ and of the preimages $x,y\in\cal O_{M}$, we may choose the uniformizer to be $\pi_M:=\sqrt[d]{\pi}$ and we may always assume that $x,y\in K$ since $K$ and $M$ have the same residue field.

If $i=0$, we see that
\[\res_{M/K}(\{y,\pi\})=\{y,\pi_M^d\}=d\{y,\pi_M\}.\]
Since $(d,p)=1$, we get the surjectivity in this case.

If $(i,p)=1$, then we have
\[\res_{M/K}(\{1+\omega x,y\})=\{1+\omega x,y\},\]
with $v_{M}(\omega)=di$, so we are done as well.

If $p|i$ with $i>0$ and $i<\epsilon_K$, then we have
\[\res_{M/K}(\{1+\varpi^p x,\pi\})=\{1+\varpi^p x,\pi_M^d\}=d\{1+\varpi^p x,\pi_M\},\] with $v_{M}(\varpi)=di/p$, and we are also done since $(d,p)=1$.

If $i=\epsilon_K$, then $(\zeta_p-1)^p=\pi^{\epsilon_K}u_0$ for some $u_0\in\cal O_{M}^\times$. Since $M/K$ is totally ramified, we may fix a unit $u\in\cal O_K^\times$ such that $\bar u_0=\bar u$. Then
\[\res_{M/K}(\{1-\pi^{\epsilon_K}u x,\pi\})=\{1-(\zeta_p-1)^p u_0^{-1}ux,\pi_M^d\}=d\{1-(\zeta_p-1)^px,\pi_M\},\]
where the last equality follows from the isomorphism \eqref{eq Kato iv} and the fact that $\bar x=\bar x\bar u\bar u_0^{-1}$. Since $(d,p)=1$, we are done in this case as well.\\

We now prove the injectivity of the restriction using corestriction, which corresponds to the norm map $N_{M/K}:\kk_2(M)\to\kk_2(K)$. Since $(d,p)=1$, we know that the composite $\cores\circ\res$ is a bijection on every $u_2^i(K)$. In particular, we know that it defines a bijection on the quotient $u_2^i(K)/u_2^{i+1}(K)$ and thus $N_{M/K}$ will correspond to the inverse of $\res_{M/K}$ up to multiplication by $d$ once we prove that it is bijective. Since $\cores\circ\res$ is a bijection and $\res$ is surjective on $u_2^{di}(M)/u_2^{di+1}(M)$, it will suffice to prove that $\cores\circ\res$ factors through $u_2^{di}(M)/u_2^{di+1}(M)$, which amounts to proving that $N_{M/K}$ sends $u_2^{di+1}(M)$ to $u_2^{i+1}(K)$. Note that this is evident for $i=\epsilon_K$ since $di=\epsilon_{M}$ and hence $u_2^{di+1}(M)=0$ by Lemma \ref{lem vanishing above eps}. Since $u_2^{\lfloor\epsilon_K\rfloor +1}(K)$ is trivial by the same result, we may assume henceforth that $i<\epsilon_K$.

Keeping our choice of $\pi_M=\sqrt[d]{\pi}$, we know by the isomorphisms \eqref{eq Kato ii} and \eqref{eq Kato iii} that 
any element in $u_2^{di+1}(M)$ can be written as a sum of symbols of the form
\begin{itemize}
\item $\{1-\pi_M^{j}x_j,y_j\}$ (if $(j,p)=1$, taking $\omega=\pi_M^j$ and replacing $x_j$ by $-x_j$);
\item $\{1-\pi_M^{j}x_j,\pi_M\}$ (if $p\mid j$, taking $\varpi=\pi_M^{j/p}$ and replacing $x_j$ by $-x_j$);
\item and an element in $u_2^{di+d}(M)$;
\end{itemize}
where $di+1\leq j\leq di+d-1$ and $x_j,y_j\in \cal O_K$. Since we have already proved that the restriction is surjective on the quotients $u_2^j(M)/u_2^{j+1}(M)$ for any $j$, we know that $u_2^{di+d}(M)$ is sent to $u_2^{i+1}(K)$ by the corestriction, so we may focus on the remaining symbols.

Let us compute the norm of an element of the form $1-\pi_M^{j}x$ with $x\in\cal O_K$. Write $e=(j,d)$, $d=ef$ and $j=eg$. Then the conjugates of $\pi_M^j$ are precisely the $\zeta_{f}^{k}\pi_M^j$ for $0\leq k\leq f-1$. Thus, we get
\[N_{M/K}(1-\pi_M^{j}x)=\left(\prod_{k=0}^{f-1}(1-\zeta_f^k\pi_M^j x)\right)^e=(1-\pi_M^{fj}x^f)^e=(1-\pi_M^{feg}x^f)^e=(1-{\pi}^{g}x^f)^e.\]

Consider now a symbol of the form $\{1-\pi_M^{j}x,\pi_M\}$ with $x\in\cal O_K$ and $p|j$. Then, if we fix an $m$ such that $md\equiv 1\pmod p$, we have
\[\{1-\pi_M^{j}x,\pi_M\}=\{1-\pi_M^{j}x,\pi_M^{md}\}=\{1-\pi_M^{j}x,\pi^{m}\}
=m\{1-\pi_M^{j}x,\pi\}.\]
Thus, taking the corestriction we get
\begin{multline}\label{eq cores p|j}
N_{M/K}(\{1-\pi_M^{j}x,\pi_M\})= N_{M/K}(m\{1-\pi_M^{j}x,\pi\})=m\{N_{M/K}(1-\pi_M^{j}x),\pi\}\\
=m\{(1-{\pi}^{g}x^f)^e,\pi\}=me\{1-{\pi}^{g}x^f,\pi\}.
\end{multline}
On the other hand, consider a symbol in $u_2^{j}(M)$ of the form $\{1-\pi_M^{j}x,y\}$ with $x,y\in\cal O_K$ and $(j,p)=1$. Then, taking the corestriction we have
\begin{equation}\label{eq cores (j,p)=1}
N_{M/K}(\{1-\pi_M^{j}x,y\})=\{N_{M/K}(1-\pi_M^{j}x),y\}=\{(1-{\pi}^{g}x^f)^e,y\}=e\{1-{\pi}^{g}x^f,y\}.
\end{equation}
Both in \eqref{eq cores p|j} and \eqref{eq cores (j,p)=1}, we get a symbol in $u_2^{g}(K)$. Since $g=\frac{j}{e}\geq\frac{j}{d}\geq i+\frac{1}{d}>i$, we see that the symbol lies indeed in $u_2^{i+1}(K)$, as desired. This concludes the proof that \eqref{eq restriction} is an isomorphism.\\

It remains to construct the maps \eqref{eq Kato i}--\eqref{eq Kato iv} themselves and to check that they are given by the formulas in the statement. Recall that $\bar M=\bar K$, since $M/K$ is totally ramified; hence $\rho_i^M$ and $\rho_i$ have the same source, and the composites written below make sense. Having proved that \eqref{eq restriction} is an isomorphism, fix an arbitrary uniformizer $\pi\in K$. Write $\pi=\pi_M^du_1$ for some uniformizer $\pi_M\in\cal O_{M}$ and some $u_1\in\cal O_{M}^\times$.
\begin{itemize}
\item We put $\rho_0:=\res^{-1}_{M/K}\circ m_d\circ\rho_0^M$, where $m_d$ denotes multiplication by $d$ and where we use the fixed uniformizer $\pi_M$ to define $\rho_0^M$. Then we see that, for a lift $y\in\cal O_K^\times$ of $\bar y\in\bar K^\times$, we have
\[\rho_0(\bar y)=\res^{-1}_{M/K}(d\{y,\pi_M\})=\res^{-1}_{M/K}(\{y,\pi_M^d\})=\{y,\pi\},\]
where the last equality uses the fact that $\pi=\pi_M^du_1$ and that $\{y,u_1\}\in u_2^1(M)$ by the isomorphism \eqref{eq Kato i original}, together with the fact that $\kk_2(\bar K)=\kk_2(\bar M)=0$ as we noticed at the beginning of the proof.
\item For $0<i<\epsilon_K$ with $(i,p)=1$ and $\omega\in K$ such that $v_K(\omega)=i$, we set $\rho_i:=\res^{-1}_{M/K}\circ \rho_{di}^M$, where we use $\omega$ as the fixed element with $v_{M}(\omega)=di$. Then we see that, for lifts $x,y\in\cal O_K$ of $\bar x,\bar y\in\bar K$, we have
\[\rho_i\left(\bar x\frac{d\bar y}{\bar y}\right)=\res^{-1}_{M/K}(\{1+\omega x,y\})=\{1+\omega x,y\}.\]
\item For $0<i<\epsilon_K$ with $p|i$ and $\varpi\in K$ with $v_K(\varpi)=i/p$, we set $\rho_i:=\res^{-1}_{M/K}\circ m_d\circ\rho_{di}^M$, where we use $\pi_M$ as the fixed uniformizer and $\varpi$ as the fixed element with $v_{M}(\varpi)=di/p$. Then we see that, for a lift $x\in\cal O_K$ of $\bar x\in\bar K$, we have
\[\rho_i\left(\bar x\pmod{\bar K^p}\right)=\res^{-1}_{M/K}(d\{1+\varpi^p x,\pi_M\})=\res^{-1}_{M/K}(\{1+\varpi^p x,\pi_M^d\})=\{1+\varpi^p x,\pi\},\]
where the last equality uses the fact that $\pi=\pi_M^du_1$ and that $\{1+\varpi^p x,u_1\}\in u_2^{di+1}(M)$ by the isomorphism \eqref{eq Kato iii original}, together with the fact that $d:\Omega_{\bar K}^1/\ker(d)\hookrightarrow\Omega_{\bar K}^2=0$, so that $\Omega^1_{\bar K}/\ker(d)=0$.
\item For $i=\epsilon_K$ and $\omega\in K$ with $v_K(\omega)=i$, we set $\rho_i:=\res^{-1}_{M/K}\circ m_d\circ\rho_{di}^M$, where we choose $\pi_M$ as the fixed uniformizer and $\omega$ as the fixed element with $v_{M}(\omega)=di$. Note that $di=\epsilon_M$ and that $\bar K=\bar M$, so the source of $\rho_{di}^M$ is precisely
\[\bar K/(\bar K \cap \bar u \Phi(\bar M))\cong\bar M/(\bar M \cap \bar u \Phi(\bar M)),\] with $u=(\zeta_p-1)^p/\omega$. Then, for a lift $x\in\cal O_K$ of $\bar x\in\bar K$, we have
\[\rho_i(\bar x)=\res^{-1}_{M/K}(d\{1+\omega x,\pi_M\})=\res^{-1}_{M/K}(\{1+\omega x,\pi_M^d\})=\{1+\omega x,\pi\},\]
where the last equality uses the fact that $\pi=\pi_M^du_1$ and that $\{1+\omega x,u_1\}=\{1+(\zeta_p-1)^pu^{-1}x,u_1\}=0$ by the isomorphism \eqref{eq Kato iv original}, together with the fact that $H^2_p(\bar M)=0$.
\end{itemize}
This concludes the proof of the case where $L=K$.\\

We move on to the general case. Since $M=L(\zeta_p)$ and $M/L$ is totally ramified, we have $\bar L=\bar M$ and the isomorphisms \eqref{eq Kato i}--\eqref{eq Kato iv} at the level of $L$, which we denote by $\rho_i^L$.

Note that, since $L/K$ is unramified, we have $\res_{L/K}(u_2^i(K))\subseteq u_2^i(L)$. Thus there is a natural map
\[\res_{L/K}:u_2^i(K)/u_2^{i+1}(K)\to u_2^i(L)/u_2^{i+1}(L),\]
which is injective for every $0\leq i\leq\epsilon_K$ by a restriction--corestriction argument. Then the map
\[\res_{M/K}:u_2^i(K)/u_2^{i+1}(K)\to u_2^{di}(M)/u_2^{di+1}(M),\]
is also injective since \eqref{eq restriction} is an isomorphism. By Lemma \ref{lem vanishing above eps}, this is also the case for the map
\[\res_{M/K}:u_2^{\epsilon_K}(K)\to u_2^{\epsilon_M}(M),\]
when $\epsilon_K\in\Z$.\\

We define the maps $\rho_i$ as follows. Fix an arbitrary uniformizer $\pi\in K$ and note that it corresponds to a uniformizer of $L$ as well.
\begin{itemize}
\item For $i=0$, we define $\rho_0:=\res_{L/K}^{-1}\circ\rho_0^L\circ\iota$, where $\iota:\bar K\to\bar L$ denotes the inclusion and where we use the fixed uniformizer $\pi$ to define $\rho_0^L$. This is well defined since, for a lift $y\in\cal O_K^\times$ of $\bar y\in\bar K^\times$, we have
\[\rho_0^L\circ\iota(\bar y)=\{y,\pi\},\]
which is clearly in the image of $\res_{L/K}$. This $\rho_0$ coincides with the morphism in \eqref{eq Kato i}, and it is injective since $\iota$ is injective. For the surjectivity, observe that the group $\kk_{2}(K)/ u_{2}^1(K)$ is generated by elements of the form $\{y,z\}$ or $\{y,\pi\}$ with $y,z\in\cal O_K^\times$. However, the restriction to $\kk_{2}(M)/ u_{2}^1(M)$ of an element of the form $\{y,z\}$ with $y,z\in \mathcal{O}_K^\times$ is trivial by \eqref{eq Kato i original} over $M$, since $\kk_2(\bar M)=0$, and hence, by the injectivity of the restriction, the symbol $\{y,z\}$ itself is trivial in $\kk_{2}(K)/ u_{2}^1(K)$. We deduce that $\kk_{2}(K)/ u_{2}^1(K)$ is generated by elements of the form $\{y,\pi\}$ with $y\in\mathcal{O}_K^\times$, and each such element is obviously in the image of $\rho_0$.
\item For $0<i<\epsilon_K$ with $(i,p)=1$ and $\omega\in K$ such that $v_K(\omega)=i$, we define $\rho_i:=\res_{L/K}^{-1}\circ\rho_i^L\circ\iota_1$, where $\iota_1:\Omega^1_{\bar K}\to\Omega^1_{\bar L}$ denotes the natural map induced by the inclusion $\iota$ and where we use $\omega$ as the fixed element with $v_{L}(\omega)=i$. This is well defined since, for lifts $x,y\in\cal O_K$ of $\bar x,\bar y\in\bar K$, we have
\[\rho_i^L\circ\iota_1\left(\bar x\frac{d\bar y}{\bar y}\right)=\{1+\omega x,y\},\]
which is clearly in the image of $\res_{L/K}$. This $\rho_i$ coincides with the morphism in \eqref{eq Kato ii}, and it is injective since $\iota_1$ is injective. For the surjectivity, observe that the group $u_{2}^i(K)/ u_{2}^{i+1}(K)$ is generated by elements of the form $\{1+\omega x,y\}$ and $\{1+\omega x,\pi\}$ with $x,y\in\cal O_K^\times$. However, for each element of the form $\{1+\omega x,\pi\}$ with $x\in\mathcal{O}_K^\times$, if we choose an $r$ such that $ri+1\equiv 0 \bmod p$, we can rewrite the symbol as
\[\{1+\omega x,\pi\}=\{1+\omega x,(-\omega x)^r\pi\}=\{1+\omega x,u\pi^{ri+1}\}=\{1+\omega x,u\},\]
for some unit $u\in\cal O_K^\times$. Hence $u_{2}^i(K)/ u_{2}^{i+1}(K)$ is generated by elements of the form $\{1+\omega x,y\}$ with $x,y\in\cal O_K^\times$, and each such element is obviously in the image of $\rho_i$.

\item For $0<i<\epsilon_K$ with $p|i$ and $\varpi\in K$ with $v_K(\varpi)=i/p$, we set $\rho_i:=\res^{-1}_{L/K}\circ\rho_{i}^L\circ\iota_0$, where $\iota_0:\bar K/\bar K^p=\Omega_{\bar K}^0/\ker(d)\to\Omega_{\bar L}^0/\ker(d)= \bar L/\bar L^p$ denotes the natural map induced by the inclusion $\iota$ and where we use $\pi$ as the fixed uniformizer and $\varpi$ as the fixed element with $v_{L}(\varpi)=i/p$. This is well defined since, for a lift $x\in\cal O_K$ of $\bar x\in\bar K$, we have
\[\rho_{i}^L\circ\iota_0\left(\bar x\pmod{\bar K^p}\right)=\rho_{i}^L\left(\bar x\pmod{\bar L^p}\right)=\{1+\varpi^p x,\pi\},\]
which is clearly in the image of $\res_{L/K}$. This $\rho_i$ coincides with the morphism in \eqref{eq Kato iii}, and since $\bar L/\bar K$ is separable, the map $\iota_0$ is injective and hence so is $\rho_i$. For the surjectivity, observe that the group $u_{2}^i(K)/ u_{2}^{i+1}(K)$ is generated by elements of the form $\{1+\varpi^p x,y\}$ and $\{1+\varpi^p x,\pi\}$ with $x,y\in\cal O_K^\times$. However, each element of the form $\{1+\varpi^p x,y\}$ with $x,y\in\mathcal{O}_K^\times$ is trivial since it becomes trivial in $u_{2}^{di}(M)/ u_{2}^{di+1}(M)$ by \eqref{eq Kato iii original} over $M$, the corresponding differential form $\bar x\,\frac{d\bar y}{\bar y}$ being zero in $\Omega^1_{\bar M}/\ker(d)\subseteq\Omega^2_{\bar M}=0$, and the restriction map $u_{2}^i(K)/ u_{2}^{i+1}(K)\rightarrow u_{2}^{di}(M)/ u_{2}^{di+1}(M)$ is injective. We deduce that $u_{2}^i(K)/ u_{2}^{i+1}(K)$ is generated by elements of the form $\{1+\varpi^p x,\pi\}$ with $x\in \mathcal{O}_K^\times$, and each such element is obviously in the image of $\rho_i$.

\item For $i=\epsilon_K$ and $\omega\in K$ with $v_K(\omega)=i$, we set $\rho_i:=\res^{-1}_{L/K}\circ\rho_{i}^L\circ\bar\iota$, where $\bar\iota:\bar K/(\bar K\cap \bar u\Phi(\bar M))\to\bar L/(\bar L\cap \bar u\Phi(\bar M))$, with $u=(\zeta_p-1)^p/\omega$, denotes the natural map induced by the inclusion $\iota$, and where we use $\pi$ as the fixed uniformizer and $\omega$ as the fixed element with $v_{L}(\omega)=i$. This is well defined since, for a lift $x\in\cal O_K$ of $\bar x\in\bar K$, we have
\[\rho_{i}^L\circ\bar\iota(\bar x)=\rho_{i}^L(\bar x)=\{1+\omega x,\pi\},\]
which is clearly in the image of $\res_{L/K}$. This $\rho_i$ coincides with the morphism in \eqref{eq Kato iv}, and since the map $\bar\iota$ is injective, so is $\rho_i$. For the surjectivity, observe that the group $u_{2}^i(K)$ is generated by elements of the form $\{1+\omega x,y\}$ and $\{1+\omega x,\pi\}$ with $x,y\in\cal O_K^\times$. However, each element of the form $\{1+\omega x,y\}$ with $x,y\in\mathcal{O}_K^\times$ is trivial since it becomes trivial in $u_{2}^{di}(M)$ by \eqref{eq Kato iv original} over $M$ and the fact that $H_p^2(\bar M)=0$, and the restriction map $u_{2}^i(K)\rightarrow u_{2}^{di}(M)$ is injective. We deduce that $u_{2}^i(K)$ is generated by elements of the form $\{1+\omega x,\pi\}$ with $x\in \mathcal{O}_K^\times$, and each such element is obviously in the image of $\rho_i$.
\end{itemize}
This concludes the proof in the general case.
\end{proof}

Now that we have extended Kato's isomorphisms, we can follow the same proof structure as for \cite[Prop.~2.8]{LieblichParimalaSuresh}, which we used in the case where $K$ was assumed to contain all the $p$-th roots of unity. Indeed, to do so, one may settle the following three ``absorption lemmas'' that are analoguous to Lemmas 1.5, 1.6 and 1.7 of loc.~cit.~and that will be helpful to inductively reduce every Milnor $\mathrm{K}$-theory class to a single symbol. Throughout, $K$ is as in Proposition \ref{prop isom Kato}: a complete discretely valued field of characteristic $0$ whose residue field $\bar K$ has dimension $1$ and characteristic $p$, and $\epsilon_K=\frac{p\,v_K(p)}{p-1}$.

Let us also record once and for all the following convention, which we use in the three proofs below. For $j\geq 1$ and a uniformizer $\pi$ of $K$, the elements $1+\omega x$ occurring in \eqref{eq Kato ii} with $v_K(\omega)=j$ and $x\in\cal O_K$ are exactly the elements $1-\pi^j z$ with $z\in\cal O_K$ (take $\omega:=-\pi^j$); likewise, the elements $1+\varpi^p x$ occurring in \eqref{eq Kato iii} with $v_K(\varpi)=j/p$ and $x\in\cal O_K$ are exactly the elements $1-\pi^j z$ with $z\in\cal O_K$ (take $\varpi:=\pi^{j/p}$ and replace $x$ by $-x$). We shall systematically write these elements in the latter form.

\begin{lemma}\label{lema absorcion i coprimo}
Let $i\geq 1$ with $(i,p)=1$ and let $i<j<\epsilon_K$. Let $a=\{1-\pi^ix,y\}$ be a symbol in $u_2^i(K)\smallsetminus u_2^{i+1}(K)$, where $\pi$ is a uniformizer in $K$ and $x,y\in\cal O_K^\times$, and let $b$ be a class in $u_2^j(K)$. Then there exist a uniformizer $\pi'\in K$ and $x',y'\in\cal O_K^\times$ such that
\[a+b=\{1-\pi'^ix',y'\}+c,\text{ with } c\in u_2^{j+1}(K).\]
\end{lemma}

\begin{proof}
Since $a\in u_2^i(K)\smallsetminus u_2^{i+1}(K)$, we know by the isomorphism \eqref{eq Kato ii} that $d\bar y\neq 0\in\Omega_{\bar K}^1$. This implies that $\bar y\not\in \bar K^p$ and hence $\bar y$ generates a $p$-basis of $\bar K$. In particular, every element in $\Omega_{\bar K}^1$ has the form $\bar z \frac{d\bar y}{\bar y}$ with $\bar z\in\bar K$.

Assume first that $(j,p)=1$. Then the isomorphism \eqref{eq Kato ii} and the observation above tell us that $b$ may be written as
\[b=\{1-\pi^jz,y\}+c,\text{ with } z\in\cal O_K \text{ and } c\in u_2^{j+1}(K).\]
We see then that $a+b=\{1-\pi^ix',y\}+c$ with $x':=x+\pi^{j-i}z-\pi^jxz$, which is a unit since $j>i>0$, so we are done in this case.\\

Assume now that $p|j$. Fix a $k$ such that $ik\equiv 1\pmod p$, so that
\[a=\{1-\pi^ix,y\}=\{1-\pi^ix,\pi^{ik}x^ky\}=\{1-\pi^ix,\pi x^k y\}=\{1-{\pi'}^iz,\pi'\},\]
where $\pi':=\pi x^k y$ is a uniformizer (recall that $x,y\in\cal O_K^\times$) and $z:=\pi^i x{\pi'}^{-i}\in\cal O_K^\times$. Then the isomorphism \eqref{eq Kato iii} applied to the uniformizer $\pi'$ tells us that $b$ may be written as
\[b=\{1-{\pi'}^jz',\pi'\}+c,\text{ with } z'\in\cal O_K \text{ and } c\in u_2^{j+1}(K).\]
We see then that $a+b=\{1-{\pi'}^ix',\pi'\}+c$ with $x':=z+{\pi'}^{j-i}z'-{\pi'}^jzz'\in\cal O_K^\times$. In order to conclude, it suffices to note that
\[\{1-{\pi'}^ix',\pi'\}=\{1-{\pi'}^ix',\pi'^{-ki+1}{x'}^{-k}\}=\{1-{\pi'}^ix',{x'}^{-k}\},\]
so that putting $y':={x'}^{-k}\in\cal O_K^\times$ we are done.
\end{proof}

\begin{lemma}\label{lema absorcion i divide a p}
Let $i\geq 1$ with $p|i$ and let $i<j<\epsilon_K$. Let $a=\{1-\pi^ix,\pi\}$ be a symbol in $u_2^i(K)\smallsetminus u_2^{i+1}(K)$, where $\pi$ is a uniformizer in $K$ and $x\in\cal O_K^\times$, and let $b$ be a class in $u_2^j(K)$. Then there exist a uniformizer $\pi'\in K$ and $x'\in\cal O_K^\times$ such that
\[a+b=\{1-\pi'^ix',\pi'\}+c,\text{ with } c\in u_2^{j+1}(K).\]
\end{lemma}

\begin{proof}
Assume first that $p\mid j$. Then the isomorphism \eqref{eq Kato iii} tells us that $b$ may be written as
\[b=\{1-{\pi}^jz,\pi\}+c,\text{ with } z\in\cal O_K \text{ and } c\in u_2^{j+1}(K).\]
We see then that $a+b=\{1-\pi^ix',\pi\}+c$ with $x':=x+\pi^{j-i}z-\pi^jxz$, which is a unit since $j>i>0$, so we are done in this case.\\

Assume now that $(j,p)=1$. Since $a\in u_2^i(K)\smallsetminus u_2^{i+1}(K)$, we know by the isomorphism \eqref{eq Kato iii} that $\bar x\not\in\bar K^p$. This implies that $\bar x$ generates a $p$-basis of $\bar K$. In particular, every element in $\Omega_{\bar K}^1$ has the form $\bar z \frac{d\bar x}{\bar x}$ with $\bar z\in\bar K$. The isomorphism \eqref{eq Kato ii} then tells us that $b$ corresponds to a differential of the form
\[\bar z\frac{d\bar x}{\bar x}=\sum_{l=0}^{p-1}(\bar z_l^p \bar x^l)\frac{d\bar x}{\bar x},\text{ with } \bar z_l\in \bar K,\]
and hence it may be written as
\[b=\sum_{l=0}^{p-1}\{1-{\pi}^jz_l^px^l,x\}+c,\text{ with } z_l\in\cal O_K \text{ and } c\in u_2^{j+1}(K).\]
If $l\neq 0$ and $m$ is such that $ml\equiv -1\pmod p$, we see that
\begin{multline*}
\{1-{\pi}^jz_l^px^l,x\}=\{1-{\pi}^jz_l^px^l,{\pi}^{mj}z_l^{mp}x^{ml+1}\}=\{1-{\pi}^jz_l^px^l,\pi^{mj}\}\\
=mj\{1-{\pi}^jz_l^px^l,\pi\}=\{(1-{\pi}^jz_l^px^l)^{mj},\pi\}=\{1-{\pi}^jw_l,\pi\},
\end{multline*}
for some $w_l\in\cal O_K$. Each of these terms can be absorbed by $a$ in the exact same way as we did in the case where $p\mid j$. Thus, we may assume that $b=\{1-\pi^jz_0^p,x\}+c$.

Set $\pi':=\pi x$. Since $p|i$, on the one hand we have
\[a=\{1-\pi^ix,\pi\}=\{1-\pi^ix,\pi^{i+1}x\}=\{1-\pi^ix,\pi x\}=\{1-{\pi'}^ix^{1-i},\pi'\}.\]
On the other hand, for $n$ such that $nj\equiv 1\pmod p$,
\[\{1-\pi^jz_0^p,x\}=\{1-\pi^j z_0^p,\pi^{nj}z_0^{np}x\}=\{1-\pi^j z_0^p,\pi x\}=\{1-{\pi'}^j x^{-j}z_0^p,\pi'\}.\]
We see then that $a+b=\{1-{\pi'}^ix',\pi'\}+c$ with $x':=x^{1-i}+{\pi'}^{j-i}x^{-j}z_0^p-{\pi'}^jx^{1-i-j}z_0^p\in\cal O_K^\times$, so we are done in this last case.
\end{proof}

\begin{lemma}\label{lema absorcion i=0}
Let $0<j<\epsilon_K$. Let $a=\{y,\pi\}$ be a symbol in $\kk_2(K)\smallsetminus u_2^{1}(K)$, where $\pi$ is a uniformizer in $K$ and $y\in\cal O_K^\times$, and let $b$ be a class in $u_2^j(K)$. Then there exist a uniformizer $\pi'\in K$ and $y'\in\cal O_K^\times$ such that
\[a+b=\{y',\pi'\}+c,\text{ with } c\in u_2^{j+1}(K).\]
\end{lemma}

\begin{proof}
Assume first that $p\mid j$. Then the isomorphism \eqref{eq Kato iii} tells us that $b$ may be written as
\[b=\{1-{\pi}^jz,\pi\}+c,\text{ with } z\in\cal O_K \text{ and } c\in u_2^{j+1}(K).\]
We see then that $a+b=\{y',\pi\}+c$ with $y':=y(1-{\pi}^jz)\in\cal O_K^\times$, so we are done in this case.\\

Assume now that $(j,p)=1$. Since $a\in \kk_2(K)\smallsetminus u_2^{1}(K)$, we know by the isomorphism \eqref{eq Kato i} that $\bar y$ is nontrivial in $\kk_1(\bar K)$ and hence $\bar y\not\in \bar K^p$. Thus, as in the last proof, we may assume that every element in $\Omega_{\bar K}^1$ has the form $\bar z \frac{d\bar y}{\bar y}$ with $\bar z\in\bar K$. The isomorphism \eqref{eq Kato ii} then tells us that $b$ can be written as
\[b=\{1-\pi^jz,y\}+c,\text{ with } z\in\cal O_K \text{ and } c\in u_2^{j+1}(K).\]
Finally, noting that
\[\{1-\pi^jz,y\}=-\{y,1-\pi^jz\}=\{y,(1-\pi^jz)^{-1}\}\]
by the anticommutativity of the Milnor symbol, cf.~\cite[\S 7.1]{GS}, we see that $a+b=\{y,\pi'\}+c$ with $\pi':=\pi(1-{\pi}^jz)^{-1}$, so we are done in this case as well.
\end{proof}

Let us now finish the proof of Theorem \ref{thm period index} in the case when $\mathrm{char}(\bar K)=p$. To do so, consider a nontrivial class $\alpha\in\kk_2(K)$. Denoting $u_2^0(K):=\kk_2(K)$ and recalling from Lemma \ref{lem vanishing above eps} that $u_2^{\lfloor\epsilon_K\rfloor+1}(K)$ is trivial, we know that there exists an integer $0\leq i\leq \epsilon_K$ such that $\alpha\in u_2^i(K)\smallsetminus u_2^{i+1}(K)$.\\

If $i=\epsilon_K$, then the isomorphism \eqref{eq Kato iv} already tells us that our class is a symbol.\\

If $i<\epsilon_K$, then we fix a uniformizer $\pi$ and put $\omega:=-\pi^i$ if $(i,p)=1$ and $\varpi:=\pi^{i/p}$ if $p|i$. We may apply one of the isomorphisms \eqref{eq Kato i}, \eqref{eq Kato ii} or \eqref{eq Kato iii} and see that $\alpha=a+b$, where $a$ is a symbol in $u_2^i(K)\smallsetminus u_2^{i+1}(K)$ given by the corresponding isomorphism and $b\in u_2^{i+1}(K)$. We may then inductively apply one of the Lemmas \ref{lema absorcion i coprimo}, \ref{lema absorcion i divide a p} or \ref{lema absorcion i=0} to express $\alpha=a+c$, with $a$ a symbol of the same shape as above, but for a possibly different uniformizer, and $c\in u_2^{\lceil\epsilon_K\rceil}(K)$. If $\epsilon_K\not\in\Z$, then we are done since $u_2^{\lceil\epsilon_K\rceil}(K)=u_2^{\lfloor\epsilon_K\rfloor+1}(K)=0$ in this case by Lemma \ref{lem vanishing above eps}. Otherwise, we have two cases.

Assume first that $p|i$ (including $i=0$). Then $a=\{y,\pi'\}$ for some $y\in\cal O_K$ and some uniformizer $\pi'\in K$. On the other hand, the isomorphism \eqref{eq Kato iv}, applied to the uniformizer $\pi'$, tells us that $c=\{w,\pi'\}$ for some $w\in 1+\pi^{\epsilon_K}\cal O_K$. We then have
\[\alpha=a+c=\{yw,\pi'\}.\]
Assume now that $(i,p)=1$. Then $a=\{1-\pi''^ix,y\}$ for some $x,y\in\cal O_K^\times$ and a uniformizer $\pi''\in K$. In particular, for an $m$ such that $im\equiv 1\pmod p$,
\[a=\{1-{\pi''}^ix,y\}=\{1-{\pi''}^ix,{\pi''}^{im}x^my\}=\{1-{\pi''}^ix,\pi''x^my\}.\]
Setting $\pi':=\pi'' x^my$ and applying the isomorphism \eqref{eq Kato iv} to this uniformizer, we see that $c=\{w,\pi'\}$ for some $w\in 1+\pi^{\epsilon_K}\cal O_K$. We then have
\[\alpha=a+c=\{(1-\pi''^ix)w,\pi'\},\]
which concludes the proof of Theorem \ref{thm period index}.

\paragraph{Fields of positive characteristic.} We close this section with some remarks on the case where $K$ itself has positive characteristic.\\

Let $K$ be a complete discretely valued field of positive characteristic with residue field $\bar K$ of dimension $\leq 1$ (a fortiori of the same characteristic). By the Cohen structure theorem, in the form given in \cite[IX, \S 3.3, Thm.~2.(c)]{Bourbaki}, the field $K$ is isomorphic to $\bar K((t))$. The proof of Theorem \ref{thm period index} given above in the cases where either $\dim(\bar K)=0$ or $\mathrm{char}(\bar K)\neq p$, which correspond to the first paragraphs of this section, work just as well in this setting. The remaining case is then that of a field $K$ of characteristic $p$ with residue field $\bar K$ of dimension $1$. In that case we only have partial results, which we now describe.

By \cite[Thm.~9.2.4, Prop.~9.2.5]{GS} we have an isomorphism
\[\Omega_K^1/(d(\Omega_K^0)+(\gamma-1)\Omega_K^1)\xrightarrow{\sim}\Br(K)[p],\]
where $\gamma$ is the inverse Cartier operator, that sends the class of $xdy$ to that of the cyclic algebra $[xy,y)$; here $[a,b)$ denotes, for $a\in K$ and $b\in K^\times$, the $K$-algebra generated by two elements $u,v$ subject to $u^p-u=a$, $v^p=b$ and $uv=v(u+1)$, cf.~\cite[\S 2.5]{GS}. Since $\dim(\bar K)=1$, the field $\bar K$ admits a $p$-basis consisting of a single element $y\in\bar K$. This tells us that $y$ and $t$ form a $p$-basis for $K$, and thus every element in $\Omega_K^1$ is a sum $\lambda dy+\mu dt$ with $\lambda,\mu\in K$. However, it seems nontrivial to prove that the quotient on the left hand side allows us to reduce such a sum to a single term, which would yield the desired result.\\

In a somewhat surprising way, such a computation is actually quite simple when $p=2$, to the point that it works for \emph{any} field $K$ of characteristic $2$ such that $[K:K^2]=2^2$.

\begin{proposition}\label{prop car 2}
Let $K$ be a field of characteristic $2$ with $[K:K^2]=2^2$. Then every element in $\mathrm{Br}(K)[2]$ is represented by a quaternion algebra.
\end{proposition}

\begin{proof}
Recall that the inverse Cartier operator $\gamma:\Omega^1_K\to\Omega^1_K/d(\Omega^0_K)$ is the additive map determined by $\gamma(da)=a^{p-1}da$ for $a\in K$, cf.~\cite[\S 9.2]{GS}. In particular $(\gamma-1)(z^2\,y\,dy)=z^2(y^2-y)dy$ when $p=2$, which is the only property that we use below.

Let $(y,t)$ be a $2$-basis of $K$ and take a differential $\omega =\lambda dy+\mu dt \in \Omega^1_K$. Write
    \begin{gather*}
        \lambda= \lambda_0^2+\lambda_1^2 y+\lambda_2^2 t+ \lambda_3^2 yt,\\
        \mu= \mu_0^2+\mu_1^2 y+\mu_2^2 t+ \mu_3^2 yt,
    \end{gather*}
    with $\lambda_0,\dots,\lambda_3,\mu_0,\dots,\mu_3\in K$. If $\mu_3=0$, then
    \begin{align*}
        \omega&=(\lambda_0^2+\lambda_1^2 y+\lambda_2^2 t+ \lambda_3^2 yt)dy+(\mu_0^2+\mu_1^2 y+\mu_2^2 t)dt\\
        &= (\lambda_0^2+\lambda_1^2 y+\lambda_2^2 t+ \lambda_3^2 yt)dy+\mu_0^2 dt+\mu_1^2 (d(yt)-tdy)+\mu_2^2 t dt\\
        &\equiv ((\lambda_2^2-\mu_1^2) t+ \lambda_3^2 yt)dy \mod (d(\Omega^0_K)+(\gamma-1)\Omega^1_K).
    \end{align*}
    If $\mu_3\neq 0$, then
    \begin{align*}
        \omega&=(\lambda_0^2+\lambda_1^2 y+\lambda_2^2 t+ \lambda_3^2 yt)dy+(\mu_0^2+\mu_1^2 y+\mu_2^2 t+ \mu_3^2 yt)dt\\
        &= \lambda_0^2 dy+\lambda_1^2 y dy+\lambda_2^2(d(ty)-ydt)+ \lambda_3^2 yt dy+(\mu_0^2+\mu_1^2 y+\mu_2^2 t+ \mu_3^2 yt)dt\\
        &\equiv \lambda_3^2 yt dy+((\mu_1^2-\lambda_2^2) y+ \mu_3^2 yt)dt \mod (d(\Omega^0_K)+(\gamma-1)\Omega^1_K)\\
        &\equiv yt d(\lambda_3^2 y+\mu_3^2 t)+(\mu_1^2-\lambda_2^2)\frac{ y d(\lambda_3^2 y+\mu_3^2 t)-\lambda_3^2 y dy}{\mu_3^2}\mod (d(\Omega^0_K)+(\gamma-1)\Omega^1_K)\\
        &\equiv \left(yt+(\mu_1^2-\lambda_2^2)\mu_3^{-2}y \right)d(\lambda_3^2 y+\mu_3^2 t) \mod (d(\Omega^0_K)+(\gamma-1)\Omega^1_K).
    \end{align*}
    In both cases we deduce that $\omega$ corresponds to a quaternion algebra in $\mathrm{Br}(K)[2]$.
\end{proof}

\section{Homogeneous spaces}\label{sec Serre II}

In this section we are interested in the arithmetic of homogeneous spaces over the field $K$. We prove in particular the following instance of Serre's Conjecture II.

\begin{theorem}\label{thm Serre II}
Let $K$ be a complete discretely valued field of dimension $2$. Let $G$ be a semisimple simply connected $K$-group. Then $H^1(K,G)=1$.
\end{theorem}

\begin{proof}[Proof of Theorem \ref{thm Serre II} when $\car(K)=0$]
As recalled in \S\ref{sec intro}, basic reductions boil the statement down to the case when $G$ is absolutely almost $K$-simple. The types $A$, $B$, $C$ and $D$ (except trialitarian $D_4$) are settled in \cite{BP} when $\bar K$ is perfect and in \cite{BFT} in general; the types $F_4$ and $G_2$ follow from \cite[Cor.~6.0.3]{GilleLNM}. By Proposition \ref{prop csa1} \eqref{item per=ind} and \cite[Thm.~8.1.5, 8.2.4, 8.3.2]{GilleLNM}, Theorem \ref{thm Serre II} also holds for semisimple simply connected groups of types $D_4$, $E_6$ and $E_7$. It remains to deal with the case of $E_8$. Let $G_0$ be the split $K$-group of type $E_8$.
Since $G_0=\mathrm{Aut}(G_0)$, it is enough to establish that
$H^1(K,G_0)=1$. Since $K$ has characteristic zero, \cite[Lem.~9.3.2]{GilleLNM} shows that it suffices to prove Proposition \ref{prop Bogomolov} below.\let\qed\relax
\end{proof}

\begin{proposition}\label{prop Bogomolov}
Let $K$ be a complete discretely valued field of characteristic $0$ and dimension $2$. Then the $2$-dimension, the $3$-dimension and the $5$-dimension of the maximal solvable extension $L$ of $K$ of degree $2^a3^b5^c$ are all at most $1$.
\end{proposition}

\begin{proof}
We denote by $p$ the characteristic exponent of $\bar K$ and we fix an $\ell\in \{2,3,5\}$. By Lemma \ref{lem_skip}, in order to prove that $\mathrm{dim}_\ell(L)\leq 1$, we need to prove that the $\ell$-primary torsion $\Br(M)\{\ell\}$ of $\Br(M)$ is trivial for every finite extension $M/L$. Consider then such an extension and fix a class $\alpha\in\Br(M)\{\ell\}$. There exists a finite extension $M_0/K$ such that $M=M_0L$ and $\alpha$ comes from an element $\alpha_0$ in $\Br(M_0)\{\ell\}$. Now, since $\bar K$ has dimension $1$ by Example \ref{examples_dim} (e), we may fix an element $\bar y$ such that $\bar K=\bar K^p(\bar y)$ and a lift $y\in K$ of $\bar y$.  We may also fix a uniformizer $\pi$ of $K$. Then, for each $n \geq 1$, the Galois closure of the extension $K_n:=K(\sqrt[\ell^n]{y},\sqrt[\ell^n]{\pi})/K$ is solvable and its degree divides $\ell^{2n}\varphi(\ell^n)=(\ell-1)\ell^{3n-1}$, which is a number whose prime factors all belong to $\{2,3,5\}$. We deduce that $K_n$ is contained in $L$ for any $n\in\N$. We consider the fields $K'_\infty:=\bigcup_n K(\sqrt[\ell^n]{y})$ and $K_\infty:=\bigcup_n K_n$ inside $L$. The residue field $\bar K'_\infty$ of $K'_\infty$ has dimension $\leq 1$ and is perfect if $\ell=p$. In particular, $\Br(\bar K'_\infty)\{\ell\}$ is trivial by Lemma \ref{lem_skip} (ii). On the other hand, since $K(\sqrt[\ell^n]{y})/K$ has trivial ramification index for every $n$, and since Galois cohomology commutes with filtered direct limits, applying Lemma \ref{lem unramified Brauer} to the fields $K(\sqrt[\ell^n]{y})$ and taking the limit yields an isomorphism $\Br(K'_\infty)\{\ell\}\xrightarrow{\sim} H^1(\bar K'_\infty,\Q_\ell/\Z_\ell)$. The same argument applied to a finite extension $L'_\infty/K'_\infty$ with ramification index $e_{L'_\infty/K'_\infty}$ yields the commutative diagram
\[\xymatrix{
\Br(K'_\infty)\{\ell\} \ar[r]^{\sim} \ar[d]_{\res_{L'_\infty/K'\infty}} & H^1(\bar K'_\infty,\Q_\ell/\Z_\ell) \ar[d]^{e_{L'_\infty/K'_\infty}\res_{\bar L'_\infty/\bar K'_\infty}} \\
\Br(L'_\infty)\{\ell\} \ar[r]^{\sim} & H^1(\bar L'_\infty,\Q_\ell/\Z_\ell).
}
\]
In particular, if we denote by $\ell^m$ the order of $\alpha_0$ in $\Br(M_0)$, we know that its restriction to any finite extension of $M'_0$ with ramification degree divisible by $\ell^m$ is trivial. We deduce that the restriction of $\alpha$ to $M_0K_\infty$ is trivial, and hence so is $\alpha$ since $K_\infty \subset L$.
\end{proof}

\begin{remarque}
Extending the proof of Proposition \ref{prop Bogomolov} to any prime $\ell$, we actually proved that $\dim (K^{\solv})\leq 1$, where $K^\solv$ denotes the maximal solvable extension of $K$.
\end{remarque}

\begin{proof}[Proof of Theorem \ref{thm Serre II} when $\car(K)>0$]
When $K$ has positive characteristic, the same reductions take us to the case of the split group of type $E_8$, which we denote by $G_0$. Indeed, the isotypic decomposition, Shapiro's Lemma, \cite{BFT}, \cite[Cor.~6.0.3, Thm.~8.1.5, 8.2.4, 8.3.2]{GilleLNM} and Proposition \ref{prop csa1} \eqref{item per=ind} all hold in arbitrary characteristic.

Here, the proof of Theorem \ref{thm Serre II} is a bit more involved and uses Bruhat--Tits theory. We fix a separable closure $K^{\sep}$ of $K$ and denote by $K^{\mathrm{unr}}$ (resp.\ $K^{\mathrm{tame}}$) the unramified closure (resp.\ the tamely ramified closure) of $K$ inside $K^{\sep}$. The Galois group of $K^{\mathrm{sep}}/K^{\mathrm{tame}}$ is the wild inertia group; this is a pro-$p$-group, so $K^{\mathrm{tame}}$ is $p$-special in the sense of \cite[\S1.1.1]{GilleLNM}. According to \cite[Thm.~8.4.1.(1)]{GilleLNM},  we have $H^1(K^{\mathrm{tame}},G_0)=1$. Next we have $H^1(K^{\mathrm{tame}}/K^{\mathrm{unr}},G_0)=1$ according to \cite[Prop.\ 3.1]{GilleTame}. It remains then to establish that $H^1(K^{\mathrm{unr}}/K,G_0)=1$ by the method of Bruhat and Tits \cite{BT3}. We need to check that the argument extends in this case to the situation when $k$ is not perfect.\\

We can reason at a finite level and we shall prove that $H^1(L/K,G_0)=1$ for a given finite unramified Galois extension $L/K$. We put $\Gamma:=\Gal(L/K)$ and denote by $\cal O_L/\cal O_K$ the Galois extension of valuation rings and by $\bar L/\bar K$ the residual Galois extension.

Let $\cB(G_{0,L})$ be the Bruhat--Tits building of $G_{0,L}$. It comes equipped with an action of $G_0(L) \rtimes \Gamma$ \cite[\S 4.2.12]{BT2}. Let $(B_0,T_0)$ be a Killing couple for $G_0$. The split $K$-torus $T_0$ defines an apartment $\cA(T_{0,L})$ of $\cB(G_{0,L})$ which is preserved by the action of $N_{G_0}(T_0)(L) \rtimes \Gamma$.

We are given a Galois cocycle $z :\Gamma \to G_0(L)$; it defines a section
\[u_z: \Gamma \to G(L) \rtimes \Gamma,\, \sigma \mapsto z_\sigma \sigma,\]
of the projection map $G_0(L) \rtimes \Gamma \to\Gamma$. This provides an action of $\Gamma$ on $\cB(G_{0,L})$ called the twisted action with respect to the cocycle $z$. The Bruhat--Tits fixed point theorem \cite[Prop.~3.2.4]{BT1} applied to the $\Gamma$-orbit of any point, which is finite and hence bounded, provides a point $y \in \cB(G_{0,L})$ which is fixed by the twisted action. This point belongs to an apartment and since $G_0(L)$ acts transitively on the set of apartments of $\cB(G_{0,L})$ there exists a suitable $g \in G_0(L)$ such that $g^{-1}\cdot y=x \in \cA(T_{0,L})$. We observe that $\cA(T_{0,L})$ is fixed pointwise by $\Gamma$ (for the standard action), so that $x$ is fixed under $\Gamma$. We consider the equivalent cocycle $z'_\sigma= g^{-1} \, z_\sigma \, \sigma(g)$ and compute
\[z'_\sigma \cdot x =  z'_\sigma \cdot \sigma(x)= (g^{-1} \, z_\sigma \, \sigma(g)) (\sigma(g^{-1})\cdot \sigma(y) )=  g^{-1} \cdot ( ( z_\sigma \sigma)\cdot y )=  g^{-1}\cdot y = x ,\]
where the second to last equality follows from the fact that $y$ is fixed under the twisted action.

Thus, without loss of generality, we may assume that $z_\sigma\cdot x=x$ for each $\sigma \in \Gamma$. We set $P_x= \Stab_{G_0(L)}(x)$; since $x$ is fixed by $\Gamma$, the group $P_x$ is preserved by the action of $\Gamma$. Let $\cP_x$ be the parahoric $\cO_L$-group scheme attached to $x$ as defined in \cite[\S 4.6.26]{BT2} (where it is denoted by 
$\mathfrak{G}_x$); we recall that its generic fibre is $G_{0,L}$ ({\it ibid}, Prop.~4.6.28). The $\cO_L$-group scheme $\cP_x$ is smooth affine with connected fibres
and we have  $\cP_x(\cO_L)=P_x$ ({\it ibid}, Prop.~4.6.32).
According to \cite[Prop.~4.3.7.(a)]{McNinch}, the unipotent radical $U_x$ of $\cP_x \times_{\cO_L} \bar L^{\mathrm{perf}}$ is defined over $\bar L$ and split over $\bar L$. It follows that the quotient $M_x= (\cP_x \times_{\cO_L} \bar L)/ U_x$ is a reductive $\bar L$-group.
An important point is that the action of $\Gamma$ on $\cP_x(\cO_L)$ arises from a semilinear action of $\Gamma$ on the $\cO_L$-scheme $\cP_x$ as explained in the beginning of \cite[\S2]{PrasadBT}. By Galois descent, the affine $\cO_L$-group scheme $\cP_x$ descends to an $\cO$-group scheme named  $\cP$. Similarly, $U_x$ descends to a unipotent $\bar K$-group $U$, which is $\bar K$-split in the view of \cite[Ch.~V, \S7]{Oesterle}, and $M=(\cP \times_{\cal O_K} \bar K)/U$ is reductive.

We have that $[z_\sigma]$ belongs to the image of $H^1(\Gamma, \cP(\cO_L)) \to H^1(\Gamma, G_0(L))$. Since $\cP$ is smooth and $\cO_L$ is complete, Hensel's Lemma \cite[Exp.~XXIV, Prop.~8.1]{SGA3} gives a bijection $H^1(\Gamma, \cP(\cO_L)) \cong H^1(\Gamma, P(\bar L))$ with $P:=\cP\times_{\cal O_K}\bar K$. Consider then the exact sequence $1\to U\to P\to M\to 1$, whose associated exact sequence of pointed sets in nonabelian Galois cohomology reads
\[H^1(\bar K,U)\to H^1(\bar K,P)\to H^1(\bar K,M).\]
Since $U$ is $\bar K$-split, we have $H^1(\bar K,U)=1$, so that the map on the right has trivial kernel. On the other hand, $H^1(\bar K,M)=1$ by Serre's Conjecture I (Steinberg's Theorem \cite[III, \S2.3, Thm.~1' and Rmk.~1]{SerreCohGal}, which applies since $\mathrm{sdim}(\bar K)\leq\dim(\bar K)=1$ (note that the definition of $\dim$ given in \cite[II, \S3.1]{SerreCohGal} coincides with our definition of $\mathrm{sdim}$ since $H^2_p(k)\simeq\Br(k)$ for any $k$ of characteristic $p$). Therefore $H^1(\Gamma,P(\bar L))\subseteq H^1(\bar L,P)=1$, whence $H^1(\Gamma, \cP(\cO_L))=1$ and $[z_\sigma]=1$.
\end{proof}

In view of Examples \ref{examples_dim} (b) and (d), we obtain the following statement.

\begin{corollary} \label{cor_BT}
We assume that $\bar K$ is separably closed of characteristic exponent $p\geq 1$ and that $[\bar K:\bar K^p] \leq p$. Let $G$ be a semisimple simply connected algebraic group over $K$. Then $H^1(K,G)=1$.
\end{corollary}

\begin{corollary}\label{cor_curve}
Let $k$ be a separably closed field of characteristic exponent $p\geq 1$ satisfying $[k:k^p] \leq p$. Let $G$ be a semisimple simply connected algebraic $k$-group and let $C$ be a smooth projective $k$-curve which is geometrically connected.
Then the localization map $H^1_{\mathrm{fppf}}(C,G) \to H^1(k(C),G)$ is onto.
\end{corollary} 

\begin{proof}
For each closed point $P$ of $C$, we denote by $K_P$ the completion  of $k(C)$ with respect to the valuation defined by $P$, by $\cal O_P$ its valuation ring and by $\bar K_P$ its residue field.
We are given a class $\gamma \in H^1(k(C),G)$ and want to show that it extends to $C$.
By a straightforward spreading out and patching argument, it is enough to show that the image of $\gamma$ in $H^1(K_P,G)$ comes from $H^1(\cal O_P,G)$ for each closed point $P$ of $C$.
According to \cite[IX, \S3.2, Thm.~2.(c)]{Bourbaki}, the field $K_P$ is isomorphic to $\bar K_P((t))$. Since $\bar K_P$ is separably closed and satisfies $[\bar K_P:\bar K_P^p] \leq p$, Corollary \ref{cor_BT}
shows that $H^1(\bar K_P,G)=1$. The local condition is then trivially fulfilled, so the proof is complete.
\end{proof}

The case of the projective line is specially interesting.

\begin{corollary}\label{cor_line}
Let $k$ be a  field of characteristic exponent $p\geq 1$ satisfying $[k:k^p] \leq p$.
Let $G$ be a semisimple simply connected algebraic $k$-group. Then the following hold:
\begin{enumerate}[(1)]
    \item\label{item H1 Zar} $H^1_{\mathrm{fppf}}(\mathbb{P}^1_{k^{\sep}}, G)=H^1_{\mathrm{Zar}}(\mathbb{P}^1_{k^{\sep}}, G)$ and  $H^1(k^{\sep}(t),G)=1$.
    \item\label{item H1 droite affine ks} $H^1(k^{\sep}[t],G)=1$.
    \item\label{item H1 droite affine k} $H^1(k,G)= H^1(k[t],G)$.
\end{enumerate}
\end{corollary} 

\begin{proof}
\begin{enumerate}[(1)]
\item Since $H^1(k^{\sep},G)=1$, the Grothendieck--Harder theorem \cite[Thm.~3.8.(a)]{GilleAffine} shows that $H^1_{\mathrm{fppf}}(\mathbb{P}^1_{k^{\sep}}, G)=H^1_{\mathrm{Zar}}(\mathbb{P}^1_{k^{\sep}}, G)$.
On the other hand, Corollary \ref{cor_curve} shows that the map  $H^1_{\mathrm{fppf}}( \mathbb{P}^1_{k^{\sep}}, G) \to H^1(k^{\sep}(t),G)$ is onto, so that the map $H^1_{\mathrm{Zar}}(\mathbb{P}^1_{k^{\sep}}, G) \to H^1(k^{\sep}(t),G)$ is onto. Since this map factors through $H^1_{\mathrm{Zar}}(k^\sep(t),G)$, we see that $H^1(k^{\sep}(t),G)=1$.

\item According to \cite[Cor.~3.10]{GilleAffine} we have an exact sequence of pointed sets
\[H^1_{\mathrm{Zar}}(k^{\sep}[t],G) \to  H^1_{\mathrm{fppf}}( k^{\sep}[t],G)\to H^1( k^{\sep}(t),G)\]
so that $H^1_{\mathrm{Zar}}( k^{\sep}[t],G) =  H^1_{\mathrm{fppf}}( k^{\sep}[t],G)$ in the view of \eqref{item H1 Zar}. Let $T$ be a maximal $k^{\sep}$-torus of $G$. The same statement provides a surjection
\[\Hom_{k^\sep}(T,\gm) \otimes_{\mathbb{Z}} \mathrm{Pic}(k^{\sep}[t]) \to 
H^1_{\mathrm{Zar}}( k^{\sep}[t],G).\]
Since the ring  $k^{\sep}[t]$ is principal, we have $\mathrm{Pic}(k^{\sep}[t]) =0$, so that $H^1_{\mathrm{Zar}}( k^{\sep}[t],G)=1$. Thus $H^1_{\mathrm{fppf}}(k^{\sep}[t],G)=1$.

\item Raghunathan--Ramanathan's Theorem \cite{RR} states that
\[H^1(k,G)= \ker( H^1(k[t],G) \to H^1(k^{\sep}[t],G)).\]
From \eqref{item H1 droite affine ks}, we obtain that $H^1(k,G)= H^1(k[t],G)$, as desired.
\end{enumerate}
\end{proof}

In the characteristic zero case, we obtain the following generalization of Kneser's theorem, 
by applying a result of Colliot-Th\'el\`ene, Parimala
and the first author \cite[Thm.\ 2.1]{CTGP}.

\begin{theorem}\label{thm_BT+}
Let $K$ be a complete discretely valued field of dimension $2$ and of characteristic zero. Let $G$ be a semisimple $K$-group and consider its simply connected cover
\[1 \to \mu \to \widetilde G \to G \to 1.\]
\noindent Then the following hold.
\begin{enumerate}[(a)]
    \item\label{item bijection} The boundary map $\partial:H^1(K,G) \to H^2(K,\mu)$ is bijective.
    \item\label{item isotropy} If the $K$-group $G$ is not purely of type $A$, it is isotropic. More generally, every normal smooth connected $K$-subgroup of $G$ that is not purely of type $A$ is isotropic.
\end{enumerate}
\end{theorem}

\begin{proof}
\begin{enumerate}[(a)]
\item Theorem \ref{thm Serre II} applied to all twisted forms of $\widetilde G$ implies that the boundary map is injective.

In order to establish surjectivity, assume first that $G$ is adjoint. Using the isotypic decomposition for $G$ and $\widetilde G$ \cite[Exp.~XXIV, 5.3, 5.9 and 5.10]{SGA3}, we may assume by Shapiro's Lemma that $G$ is absolutely almost $K$-simple, so that the absolute Cartan--Dynkin diagram of $G$ is connected. Surjectivity is then obvious in types $E_8$, $F_4$ and $G_2$, so we can exclude those types. Finally, since the field $K$ satisfies the ``period equals index'' property by Proposition \ref{prop csa1} \eqref{item per=ind}, the boundary map is surjective in the view of \cite[Thm.\ 2.1.(a)]{CTGP}.

We move on to the general case. We denote by $G_{\ad}$ the adjoint group of $G$. Consider the following commutative diagram with exact rows and columns
\[
\xymatrix@R=1em{
&  && 1\ar[d] \\
& 1\ar[d] && \nu \ar[d] \\
1 \ar[r] &  \mu \ar[r]\ar[d] & \widetilde G \ar[r]\ar@{=}[d] & G \ar[r] \ar[d] & 1 \\
1 \ar[r] &  \widetilde \mu \ar[r] \ar[d] & \widetilde G \ar[r] & G_{\ad} \ar[r] & 1 \\
& \nu \ar[d] && \\
& 1 &&
}
\]
This induces the commutative exact diagram 
of pointed sets

\[
\xymatrix{
H^1(K,\nu) \ar[d] \ar@{=}[r] & H^1(K,\nu) \ar[d] \\  
H^1(K,G) \ar[r]^\partial \ar[d] & H^2(K,\mu) \ar[d] \\
H^1(K,G_{\ad}) \ar[r]^{\sim} \ar[d]
 & H^2(K,\widetilde \mu)\ar[d]  \\
H^2(K,\nu) \ar@{=}[r] & H^2(K,\nu)  
}
\]
where we have used the adjoint case. We are given a $\gamma \in H^2(K,\mu)$. The diagram provides an element $\alpha \in H^1(K,G_{\ad})$ mapping to the image of $\gamma$ in $H^2(K, \widetilde \mu)$, and $\alpha$ lifts to an element $\beta \in H^1(K,G)$.

Since $H^1(K,\nu)$ acts transitively on the fibers of $H^1(K,G)\to H^1(K,G_{\ad})$ (cf.~\cite[I, \S 5.7, Prop.~42]{SerreCohGal}) and since the boundary map $\partial:H^1(K,G) \to H^2(K,\mu)$ is $H^1(K,\nu)$-equivariant, we can modify $\beta$ by an element of $H^1(K,\nu)$ in order to ensure that $\partial(\beta)=\gamma$.

\item We can assume that $G$ is adjoint and absolutely almost $K$-simple, the general case following by applying this to each isotypic factor. If $G$ is of type $E_8$, it is split in the view of Theorem \ref{thm Serre II}, so in particular isotropic. Otherwise we apply \cite[Thm.\ 2.1.(b)]{CTGP}.
\end{enumerate}
\end{proof}

We expect that the characteristic zero assumption can be removed from Theorem \ref{thm_BT+}. In the case of a local field of positive characteristic, this has been proven by Nguy\~{\^e}n Qu\'{\^o}c Th\'{\u a}ng (cf.~\cite[Thm.\ A]{Thang}).

\section{Torsors under parahoric group schemes}\label{section_KP}
Theorem \ref{thm Serre II} allows us to extend to the $E_8$-case the method of Kisin--Pappas on extension of torsors \cite[\S 1.4]{KisinPappas}. We briefly explain this in this section.

We adopt the following notation, taken from \cite{KisinPappas}. Let $p$ be a prime number, let $k$ be either a finite extension of $\F_p$ or an algebraic closure of $\F_p$, and set $W = W(k)$ (the ring of $p$-typical Witt vectors) and $K_0 = \mathrm{Frac}(W)$. We also set $\bD := \spec(W[[u]])$ and
$\bD^\times := \bD \setminus\{(u, p)\}$.

We are given a connected reductive group $G$ over $K_0$, and let $\cG$ be a parahoric Bruhat--Tits (smooth) group scheme over $W$ for $G$; i.e.\ $\cG = G_x^0$ for a point $x$ in the Bruhat--Tits building $B(G, K_0 )$.
We discuss now the following key vanishing statement.

\begin{proposition} \label{prop_KP}
We have  $H^1(\bD^\times , \cG ) = 1$.
\end{proposition}

This statement is not new. It is due to Kisin and Pappas \cite[Prop.~1.4.3]{KisinPappas} under the assumptions that $G$ splits over a tamely ramified extension of $K_0$ and that no almost simple factor of $G_{K_0^{\sep}}$ is of type $E_8$; and it has been proved in general, by a completely different method, by Ansch\"utz \cite[Prop.~10.3]{An22}.

What Theorem \ref{thm Serre II} contributes is the removal of the restriction on the type in the argument of \cite{KisinPappas}. Indeed, the hypothesis on $E_8$ is used in loc.~cit.\ only at one point, namely in order to invoke Serre's Conjecture II over the complete discretely valued field $E = K_0\{\{u\}\}$ (and over its finite extensions), whose ring of integers is the $p$-adic completion of $W[[u]]_{(p)}$ and whose residue field $k((u))$ has dimension $1$ when $k$ is an algebraic closure of $\F_p$. For the remaining types this input was available, whereas for $E_8$ it was not. Theorem \ref{thm Serre II} now supplies it for every type, and the argument of \cite[\S 1.4]{KisinPappas} therefore goes through verbatim, giving a proof of Proposition \ref{prop_KP} that is considerably shorter than the one in \cite{An22} in the tamely ramified case.

\end{document}